\documentclass[reqno]{amsart}
\usepackage{amssymb}
\usepackage{amsfonts}
\usepackage{mathrsfs}
\usepackage{bbm}
\usepackage{xcolor}
\theoremstyle{plain}
    \newtheorem{theorem}{Theorem}[section]
    \newtheorem*{theorem*}{Theorem}
    \newtheorem{lemma}[theorem]{Lemma}

\theoremstyle{definition}
    \newtheorem{definition}{Definition}[section]
    \newtheorem{remark}{Remark}[section]
    
\numberwithin{equation}{section}

\renewcommand{\r}{\right}
\DeclareMathOperator{\Span}{span}

\begin{document}
%%%%%%%%%%%%%%%%%%%%%%%%%%%%%%%%%%%%%%%%%%%%%%%%%%%%%%%%

\title[]{Long-time asymptotics of 3-solitary waves for the damped nonlinear Klein-Gordon equation}
\author[K. Ishizuka]{Kenjiro Ishizuka}
\address[K.Ishizuka]{Research Institute for Mathematical Sciences, Kyoto University, Kyoto 606-8502, JAPAN}
\email{ishizuka@kurims.kyoto-u.ac.jp}
\date{\today}
%\keywords{damped wave equation, dissipation, Strichartz estimates, energy critical}
\date{\today}
%\subjclass[2020]{35Q55,35B44 etc.}
\maketitle

\begin{abstract}
We consider the damped nonlinear Klein-Gordon equation:
\begin{align*}
\partial_{t}^2u-\Delta u+2\alpha \partial_{t}u+u-|u|^{p-1}u=0, \ & (t,x) \in \mathbb{R} \times \mathbb{R}^d,
\end{align*}
where $\alpha>0$, $1\leq d\leq 5$ and energy sub-critical exponents $p>2$. In this paper, we prove that 3-solitary waves behave as if the three solitons are on a line. Furthermore, the solitary waves have alternative signs and their distances are of order $\log{t}$.
\end{abstract}

\tableofcontents
%%%%%%%%%%%%%%%%%%%%%%%%%%%%%%%%%%%%%%%%%%

\section{Introduction}
\subsection{Setting of the problem}
We consider the following damped nonlinear Klein-Gordon equation:
\begin{align}
\label{DNKG}
\tag{DNKG}
	\left\{
	\begin{aligned}
		&\partial_{t}^2u-\Delta u+2\alpha \partial_{t}u+u-f(u)=0, & (t,x) \in \mathbb{R} \times \mathbb{R}^d,
		\\
		&\left(u(0,x),{\partial}_tu(0,x))=(u_{0}(x), u_1(x)\right)\in \mathcal{H},
	\end{aligned}
	\r.
\end{align}
where $\mathcal{H}=H^1(\mathbb{R}^d)\times L^2(\mathbb{R}^d)$, $f(u)=|u|^{p-1}u$, with a power $p$ in the energy sub-critical range, namely
\begin{align*}
2<p<\infty \ \mbox{for}\ d=1,\ 2\ \mbox{and}\ 2<p<\frac{d+2}{d-2}\ \mbox{for}\ d=3,4,5.
\end{align*}
It follows from \cite{BRS} that the Cauchy problem for \eqref{DNKG} is locally well-posed in the energy space $\mathcal{H}$. Moreover defining the energy of a solution $\vec{u}=(u, {\partial}_tu)$ by
\begin{align*}
	E(\vec{u}(t))=\frac{1}{2}\|\vec{u}(t)\|_{\mathcal{H}}^2-\frac{1}{p+1}\| u(t)\|_{L^{p+1}}^{p+1},
\end{align*}
it holds 
\begin{align}\label{energydecay}
	E(\vec{u}(t_2))-E(\vec{u}(t_1))=-2\alpha \int_{t_1}^{t_2} \| {\partial}_tu(t)\|_{L^2}^2dt.
\end{align}
In addition, the equation \eqref{DNKG} enjoys several invariances:
\begin{align}\label{invariance}
\begin{aligned}
u(t,x)&\mapsto -u(t,x),
\\
u(t,x)&\mapsto u(t,x+y),
\\
u(t,x)&\mapsto u(t+s,x),
\end{aligned}
\end{align}
where $y\in \mathbb{R}^d$ and $s\in \mathbb{R}$. In this paper, we are interested in the dynamics of multi-solitary waves related to the ground state $Q$, which is the unique positive, radial $H^1(\mathbb{R}^d)$ solution of 
\begin{align}\label{Qode}
-\Delta Q+Q-f(Q)=0.
\end{align}
The ground state generates the stationary solution $(Q,0)$. By \eqref{invariance}, The function $(-Q,0)$ as well as any translate $(Q(\cdot+y),0)$ are also solutions of \eqref{DNKG}.

There have been intensive studies on global behavior of general (large) solutions for nonlinear dispersive equations, where the guiding principle is the \textit{soliton resolution conjecture}, which claims that generic global solutions are asymptotic to superposition of solitons for large time. For \eqref{DNKG}, Feireisl \cite{F} proved the soliton resolution along a time sequence: for any global solution $\vec{u}$ of \eqref{DNKG}, there is a time sequence $t_n\to \infty$ and a set of stationary solutions $\{ \varphi_k\}_{k=1}^K$ ($K\geq 0$), and sequences $\{ c_{k,n}\}_{k=1,2,\ldots, K,n\in \mathbb{N}}$ such that
\begin{align*}
\lim_{n\to\infty} \left( \|u(t_n)-\sum_{k=1}^K\varphi_k(\cdot -c_{k,n})\|_{H^1}+\|{\partial}_tu(t_n)\|_{L^2}\right)=0,
\\
\lim_{n\to\infty} \left(\min_{j\neq k} |c_{j,n}-c_{k,n}|\right)=\infty.
\end{align*}

Recently, there have been many studies about the superposition of the ground state.

%More recently, Burq, Raugel, and Schlag \cite{BRS} proved the soliton resolution for \eqref{DNKG} in the full limit $t\to\infty$ for all radial solutions. 

\subsection{Main result}

First, we define a solution that behaves as the superposition of ground states.
\begin{definition}\label{defKsoli}
A solution $\vec{u}$ of \eqref{DNKG} is called a \textit{$K$-solitary wave} for $K\in\mathbb{N}$ if there exist $\sigma \in \{-1,1\}^K$ and $z:[0,\infty)\to (\mathbb{R}^d)^K$ such that 
\begin{align}
\label{Ksoli}
\begin{aligned}
\lim_{t\to \infty} \left(\|u(t)-\sum_{k=1}^K\sigma_k Q(\cdot-z_k(t))\|_{H^1}+\|{\partial}_tu(t)\|_{L^2}\right)&=0,
\\
\lim_{t\to\infty}\left( \min_{j\neq k}|z_j(t)-z_k(t)|\right)&=\infty.
\end{aligned}
\end{align}
\end{definition}

\begin{remark}
In space dimension $d=1$, thanks to the damping term $\alpha$, Definition \ref{defKsoli} remains equivalent to the original definition even if it is replaced with convergence in the sense of time sequences $t_n\to\infty$. On the other hand, in space dimension $d\geq 2$, even if a solution converges along a sequence of times, it is not clear whether it remains uniformly bounded in $\mathcal{H}$ for all $t$. Hence, the equivalence of the two notions is not known. In this paper, we are interested in the long-time behavior of the positions of solitons. Therefore, we adopt Definition \ref{defKsoli}.
\end{remark}

%\begin{remark}
%For \eqref{DNKG}, thanks to the damping term $\alpha$, Definition \ref{defKsoli} remains equivalent to the original definition even if it is replaced with convergence in the sense of time sequences $t_n\to\infty$. Therefore, to facilitate the subsequent discussion, we adopt Definition \ref{defKsoli}. The equivalence of the two definitions can be demonstrated using classical compactness arguments. For further details, we refer to \cite[Section 3]{I}.
%\end{remark}

In this paper, we clarify the long-time behavior of a $3$-solitary wave. In the case of $K=3$, the possible sign combinations are either ``all the same sign'' or ``one sign differing from the other two''. Furthermore, C\^{o}te and Du \cite{CD} proved that a $K$-solitary wave in which all solitary waves have the same sign does not exist. Therefore, in a $3$-solitary wave, the sign of one soliton differs from the signs of the other two. On the other hand, when $K\geq 3$, the behavior of the centers of each soliton in a $K$-solitary wave was unclear in the case of general space dimensions. Our main result is that a $3$-solitary wave behaves as $t\to\infty$, the centers of the solitary waves align along a straight line and the sign of each solitary wave differs from that of its neighboring solitary wave. 
\begin{theorem}\label{maintheorem}
If a solution $\vec{u}$ of \eqref{DNKG} is a $3$-solitary wave, there exist $\sigma=\pm 1$, $z_{\infty}\in \mathbb{R}^d$, $\omega_{\infty}\in S^{d-1}$, and a function $z_k:[0,\infty)\to \mathbb{R}^d$ for $k=1,2,3$ such that for all $t>0$,
\begin{align}\label{3soli}
\left\| u(t)-\sigma \sum_{k=1}^3 (-1)^k Q(\cdot-z_k(t))\right\|_{H^1}+\|{\partial}_tu(t)\|_{L^2}\lesssim t^{-1},
\end{align}
and for $k=1,2,3$,
\begin{align}\label{zes}
z_k(t)=z_{\infty}+(k-2)\left( \log{t}-\frac{d-1}{2}\log{(\log{t})}+c_0\right)\omega_{\infty}+O\left(\frac{\log{(\log{t})}}{\log{t}}\right),
\end{align}
where $c_0$ is a constant depending only on $d,\ \alpha$, and $p$.
\end{theorem}
C\^{o}te and Du \cite{CD} proved the existence of a 3-solitary wave in which the centers of the solitary waves align along a straight line. Our result guarantees that the asymptotic behavior of any 3-solitary wave necessarily follows the pattern described in Theorem \ref{maintheorem}.

\subsection{Previous results}

First, we introduce results related to the soliton resolution conjecture. The soliton resolution conjecture has been studied for the energy-critical wave equation. We refer to \cite{DKM1,JL4}, in which the conjecture was proven for all radial solutions in space dimension $d\geq 3$. Historically, we refer to \cite{CDKM, DJKM, DKM1, DKM2}, the conjecture has been proved in space dimension $d=3,5,7,\ldots$ or $d=4,6$ using the method of channels. In recent years, Jendrej and Lawrie \cite{JL4} proved the conjecture in space dimension $d\geq 4$ using a novel argument based on the analysis of collision intervals. Similar results regarding the conjecture have been obtained for cases with a potential \cite{JLX,LMZ}, with a damping term \cite{GZ}, as well as for the wave maps \cite{DKMM,JL1}, for the heat equation \cite{A1}, all under some rotational symmetry.

On the other hand, the conjecture remains widely open for the nonlinear Klein-Gordon equation, namely the $\alpha=0$ case. Nakanishi and Schlag \cite{NS} proved the conjecture as long as the energy of the solution is not much larger than the ground state energy in space dimension $d=3$. 
For the damped nonlinear Klein-Gordon equation case, the energy decay \eqref{energydecay} makes the analysis simpler than the undamped case. Historically, Keller \cite{K1} constructed stable and unstable manifolds in a neighborhood of any stationary solution of \eqref{DNKG}, and Feireisl \cite{F} proved the conjecture along a time sequence. Burq, Raugel and Schlag \cite{BRS} proved the conjecture in the full limit $t\to\infty$ for all radial solutions. In particular, C\^{o}te, Martel and Yuan \cite{CMY} proved the conjecture for $d=1$ without size restriction and radial assumption for a non-integrable equation without any size or symmetry restriction and with moving solitons. In the result without any size or symmetry restriction, Kim and Kwon \cite{KK} proved the conjecture for Calogero-Moser derivative nonlinear Sch\"{o}dinger equation.

 On the other hand, in the case of general dimensions $d\geq 2$, it remains unclear how solitary waves are arranged and how they move. This is because the ODE system for the centers of the solitary waves becomes more complex in general dimensions. 
 
 In general space dimension, C\^{o}te, Martel, Yuan, and Zhao \cite{CMYZ} proved that 2-solitary waves with same sign do not exist. Furthermore, they constructed a Lipschitz manifold in the energy space with codimension 2 of those solutions asymptotics to 2-solitary waves with the opposite signs. Subsequently, the author and Nakanishi \cite{IN} provided a complete classification of solutions into 5 types of global behavior for all initial data in a small neighborhood of each superposition of two ground states with the opposite signs. Additionally, results concerning an excited state of \eqref{Qode} (stationary solution) can be found in \cite{CY}. Recently, C\^{o}te and Du \cite{CD} constructed multi-solitary waves that asymptotically resemble various objects by imposing group-symmetric assumptions. This result has facilitated the study of the global behavior of multi-solitary waves in general dimensions.

A crucial aspect of analyzing multi-solitary waves is understanding the movement of their centers caused by interactions between solitary waves. In particular, a 2-solitary wave with the same sign does not exist because the interaction between them generates an attractive force. On the other hand, the author \cite{I} constructed 2-solitary waves with the same signs for the damped nonlinear Klein-Gordon equation with a repulsive delta potential in $d=1$. This is because the repulsive force of the potential dominates the attractive interaction between the solitons.

 In addition, several studies have been conducted on multi-solitary waves for the nonlinear dispersive equations. For the nonlinear Klein-Gordon equation, C\^{o}te and Mun\~{o}z \cite{CM} constructed multi-solitary waves based on the ground state. Furthermore, C\^{o}te and Martel \cite{CM1} extended the work of \cite{CM} by constructing multi-solitary waves based on the stationary solution of the nonlinear Klein-Gordon equation. Notably, multi-solitary waves usually behave asymptotically as the solitary waves moving at distinct constant speeds. More recently, Aryan \cite{A} has constructed a 2-solitary wave where the distance between each solitary wave is $2(1+O(1))\log{t}$ as $t\to \infty$. There are several results regarding multi-solitary waves with logarithmic distance; see \textit{e.g.} \cite{GI,MN,N1,N2}. Notably, for the damped nonlinear Klein-Gordon equation, multi-solitary waves of \eqref{DNKG} cannot move at a constant speed by the damping $\alpha$. \cite{J1,J2,JL1,JL2,JL4,JL3} are studies on two-bubble solutions.

\subsection{Notation}
Let $\{e_1,\ldots, e_d\}$ denote the canonical basis of $\mathbb{R}^d$, and we denote ${\partial}_k$ as the derivative in the $x_k$-direction. The $L^2$ scalar product is denoted for any pair of functions $u,v\in L^2(\mathbb{R}^d)$ by
\begin{align}\label{L2product}
\langle u,v\rangle:=\int_{\mathbb{R}^d}u(x)v(x)dx.
\end{align}
In addition, when pairing elements of $H^{-1}$ and $H^1$, we use the same notation and do not distinguish between the $L^2$ inner product and the corresponding dual pairing, since in all cases in this paper the pairings are induced by the $L^2$ structure (e.g. via integration by parts). The Euclidean inner product is denoted for any pair of vectors $x,y\in \mathbb{R}^d$ by
\begin{align}\label{Euclideanproduct}
x\cdot y=\sum_{k=1}^dx_ky_k.
\end{align}
 In this paper, $X\lesssim Y$ means that $X\leq CY$ for some constant $C>0$. $X\sim Y$ means that $X\lesssim Y$ and $Y\lesssim X$.

\section{Basic properties of the ground state}

In this section, we collect basic properties of the ground state $Q$. First, we recall the decay rate of $Q$. Since $Q$ is radial, we define $q:[0,\infty)\to (0,\infty)$ as 
\begin{align*}
Q(x)=q(|x|).
\end{align*}
Then, $q$ satisfies
\begin{align}\label{qode2}
q^{\prime \prime}+\frac{d-1}{r}q^{\prime}-q+q^p=0.
\end{align}
Furthermore, there exists a constant $c_q$ depending only on $d,\ p$ such that 
\begin{align}\label{qes}
|q(r)-c_qr^{-\frac{d-1}{2}}e^{-r}|+|q^{\prime}(r)+c_qr^{-\frac{d-1}{2}}e^{-r}|\lesssim r^{-\frac{d+1}{2}}e^{-r}.
\end{align}
Next, we recall the following symmetry property
\begin{align}\label{Qsymmetry}
\int_{\mathbb{R}^d} {\partial}_iQ{\partial}_jQ=0\ \mbox{for}\ i\neq j.
\end{align}
Furthermore, the linearized operator $L$  around $Q$ is denoted by
\begin{align}\label{Ldef}
L=-\Delta+1-pQ^{p-1}.
\end{align}
We recall standard properties of the operator $L$. For further details, we refer to \cite{CGNT}.

\begin{lemma}\label{Lproperty}
\begin{enumerate}
\item The unbounded operator $L$ on $L^2(\mathbb{R}^d)$ with domain $H^2(\mathbb{R}^d)$ is self-adjoint, its continuous spectrum is $[1,\infty)$, its kernel is 
\\
$\Span{ \{ {\partial}_kQ : k=1,2,\ldots d \}}$ and it has a unique negative eigenvalue $-\nu_0^2$ for a constant $\nu_0>0$, with a corresponding smooth, normalized, even eigenfunction $\phi$ ($\|\phi\|_{L^2}=1$). Moreover, on $\mathbb{R}^d$, for $\beta\in \mathbb{N}^d$, we have
\begin{align*}
|{\partial}_x^{\beta} \phi(x)|\lesssim e^{-\sqrt{1+\nu_0^2}|x|}.
\end{align*}
\item There exists $c_L>0$ such that, for all $\varphi \in H^1(\mathbb{R}^d)$,
\begin{align*}
\langle L\varphi,\varphi\rangle\geq c_L\|\varphi\|_{H^1}^2-\frac{1}{c_L}\left(\langle \varphi,\phi\rangle^2+\sum_{k=1}^d\langle\varphi,{\partial}_kQ\rangle^2\right).
\end{align*}
\end{enumerate}
\end{lemma}
The linearized evolution around $(Q,0)$ is given by
\begin{align}\label{lode1}
\frac{d}{dt}v= 
\begin{pmatrix}
0 & 1
\\
-L& -2\alpha
\end{pmatrix}
v.
\end{align}
We define $\nu^{\pm}$ and $Y^{\pm}$ as follows:
\begin{align*}
\nu^{\pm}&=-\alpha \pm \sqrt{\alpha^2+\nu_0^2},
\\
Y^{\pm}&=
\begin{pmatrix}
1
\\
\nu^{\pm}
\end{pmatrix}
\phi.
\end{align*}
Given that $\nu^+>0$, the solution $v=e^{\nu^+t}Y^+$ of \eqref{lode1} exhibits the exponential instability of $(Q,0)$. In particular, we see that the presence of the damping $\alpha>0$ in the equation does not remove the exponential instability.

Last, we collect estimates of the nonlinear interaction. We define $\theta_{\star}$ as a constant satisfying
\begin{align}\label{thetastar}
\theta_{\star}=\min{(p-1,2)}.
\end{align}
\begin{lemma}\label{ni}
We assume that $z_1,\ z_2\in \mathbb{R}^d$ satisfy $|z_1-z_2|\gg 1$. Then by defining
\begin{align*}
Q_1=Q(\cdot-z_1),\ Q_2=Q(\cdot-z_2),\ \phi_1=\phi(\cdot-z_1),\ \phi_2=\phi(\cdot-z_2),
\end{align*}
we have for any $0<m<m^{\prime}$,
\begin{align}
\int |Q_1Q_2|^m+\int (|\nabla Q_1||\nabla Q_2|)^m &\lesssim e^{-m^{\prime}|z_1-z_2|}, \label{si1}
\\
\int |Q_1|^m|Q_2|^{m^{\prime}}&\lesssim q(|z_1-z_2|)^m. \label{si2}
\end{align}
In particular, we have
\begin{align*}
\int |Q_1|^{p-1}\phi_2\lesssim q(|z_1-z_2|),
\end{align*}
and for any $1<\theta<\theta_{\star}$, 
\begin{align*}
\int |Q_1|^2|Q_2|^{p-1}\lesssim q(|z_1-z_2|)^{\theta}.
\end{align*}
\end{lemma}
\begin{proof}
By direct computation, we obtain these estimates.  For details, see \cite[Lemma 7]{N2}. 
\end{proof}

We define $H:\mathbb{R}^d\to \mathbb{R}^d$ as 
\begin{align}\label{Hdef}
H(z)=-\int_{\mathbb{R}^d}(\nabla Q^p(x))Q(x-z)dx.
\end{align}
Furthermore, we define $g:(0,\infty)\to \mathbb{R}$ as
\begin{align*}
H(z)=\frac{z}{|z|} g(|z|).
\end{align*}
We note that the above definition is well-defined. Then, we have for $h\in \mathbb{R}$,
\begin{align*}
g(r+h)-g(r)=-\int_{\mathbb{R}^d} {\partial}_1(Q^p)\left( Q(x-(r+h)e_1)-Q(x-re_1)\right).
\end{align*}
Therefore, we obtain
\begin{align*}
g^{\prime}(r)=\int_{\mathbb{R}^d} {\partial}_1(Q^p){\partial}_1Q(x-re_1)=-\int_{\mathbb{R}^d}Q^p(x){\partial}_1^2Q(x-re_1)dx.
\end{align*}
In addition, by direct computation, we have
\begin{align*}
{\partial}_1^2Q(x-re_1)=\frac{\sum_{k=2}^dx_k^2}{|x-re_1|^3}q^{\prime}(|x-re_1|)+\frac{(x_1-r)^2}{|x-re_1|^2}q^{\prime \prime}(|x-re_1|).
\end{align*}
By \eqref{qode2}, we have
\begin{align*}
|g(r)-c_gr^{-\frac{d-1}{2}}e^{-r}|+|g^{\prime}(r)+c_gr^{-\frac{d-1}{2}}e^{-r}|\lesssim r^{-\frac{d+1}{2}}e^{-r},
\end{align*}
where 
\begin{align*}
c_g=\int_{\mathbb{R}^d} Q^pe^{-x_1}.
\end{align*}
We will use $H$ and $g$ more precisely in Section 3.

\section{Dynamics of $K$-solitary waves}
In this section, we compute the dynamics around the sum of the ground state.

\subsection{Modulation of the center}

When we compute the dynamics around the sum of the ground state, we have some freedom for the choice of centers of the ground state. Here, to facilitate the subsequent discussion, we consider the following choice of the centers.

\begin{lemma}\label{modulation}
There exists $\delta_{mod}\in (0,1)$ such that for any $\varphi=(\varphi_1,\varphi_2) \in \mathcal{H}$,  $z\in (\mathbb{R}^d)^K$ and $\sigma \in \{-1,1\}^K$satisfying
\begin{align*}
\min_{i\neq j} |z_i-z_j|&\geq \frac{1}{\delta_{mod}},
\\
\left\| \varphi_1-\sum_{k=1}^K \sigma_k Q(\cdot-z_k)\right\|_{H^1}+\|\varphi_2\|_{L^2}&\leq \delta_{mod},
\end{align*}
there exists a unique $\tilde{z}\in (\mathbb{R}^d)^K$ satisfying $|z-\tilde{z}|\lesssim \delta_{mod}$ and
\begin{align*}
\int \left\{\varphi_2+2\alpha(\varphi_1-\sum_{k=1}^K \sigma_k Q(\cdot-\tilde{z}_k))\right\} {\partial}_lQ(\cdot-\tilde{z}_i)=0
\end{align*}
for $1\leq l\leq d$ and $1\leq i\leq K$. Moreover $\tilde{z}$ is $C^1$ with respect to $(\varphi,z)$ and we have
\begin{align*}
|z-\tilde{z}|\lesssim \left\| \varphi_1-\sum_{k=1}^K \sigma_k Q(\cdot-z_k)\right\|_{H^1}+\|\varphi_2\|_{L^2}.
\end{align*}
\end{lemma}

\begin{remark}
This lemma may appear to be expressed differently from \cite[Lemma 4.1]{IN}. However, it is equivalent to \cite[Lemma 4.1]{IN}. 
\end{remark}

\begin{proof}
See \cite[Lemma 4.1]{IN}.
\end{proof}

This lemma enables us to impose constraints on the centers of a $K$-solitary wave.

\begin{lemma}\label{modKsoli}
When a solution $\vec{u}$ satisfies \eqref{Ksoli} for some $\sigma \in \{-1,1\}^K$ and $z$, there exist a $C^1$ function $\tilde{z}:[0,\infty)\to (\mathbb{R}^d)^K$ and $T>0$ such that 
\begin{align}
\lim_{t\to\infty}(\left\| u(t)-\sum_{k=1}^K \sigma_kQ(\cdot-\tilde{z}(t))\right\|_{H^1}+\|{\partial}_tu(t)\|_{L^2})&=0,\label{modlim1}
\\
\lim_{t\to\infty} \min_{i\neq j}|\tilde{z}_i(t)-\tilde{z}_j(t)|&=\infty,\label{modlim2}
\\
\lim_{t\to\infty} |z(t)-\tilde{z}(t)|&=0,\label{modlim3}
\end{align}
and for $1\leq l\leq d$, $1\leq i\leq K$, $t>T$,
\begin{align}\label{modeq}
\int \{ {\partial}_tu(t)+2\alpha(u(t)-\sum_{k=1}^K \sigma_k Q(\cdot-\tilde{z}_k(t)))\} {\partial}_lQ(\cdot-\tilde{z}_i(t))=0.
\end{align}
\end{lemma}

\begin{proof}
By \eqref{Ksoli}, for $0<\delta<\delta_{mod}$, there exists $T>0$ such that for $t>T$
\begin{align*}
\| u(t)-\sum_{k=1}^K \sigma_kQ(\cdot-z_k(t))\|_{H^1}+\|{\partial}_tu(t)\|_{L^2}<\delta.
\end{align*}
Then, by Lemma \ref{modulation}, there exists the $C^1$ function $\tilde{z}$ such that for $t>T,\ 1\leq l\leq d,\ 1\leq i\leq K$ \eqref{modeq} is satisfied. Furthermore, we have
\begin{align*}
|z(t)-\tilde{z}(t)|\lesssim  \|u(t)-\sum_{k=1}^K \sigma_kQ(\cdot-z_k(t))\|_{H^1}+\|{\partial}_tu(t)\|_{L^2}.
\end{align*}
Thus we obtain \eqref{modlim2} and \eqref{modlim3}. We also have
\begin{align*}
\| u(t)-\sum_{k=1}^K \sigma_kQ(\cdot-\tilde{z}_k(t))\|_{H^1}&\leq \| u(t)-\sum_{k=1}^K \sigma_kQ(\cdot-z_k(t))\|_{H^1}
\\
&\quad +\sum_{k=1}^K\|Q(\cdot-z_k(t))-Q(\cdot-\tilde{z}_k(t))\|_{H^1}.
\end{align*}
By \eqref{Ksoli} and 
\begin{align*}
\sum_{k=1}^K\|Q(\cdot-z_k(t))-Q(\cdot-\tilde{z}_k(t))\|_{H^1}\lesssim |z(t)-\tilde{z}(t)|,
\end{align*}
we obtain \eqref{modlim1}. Thus we complete the proof.
\end{proof}

%By Lemma \ref{modKsoli}, to facilitate the subsequent discussion, we impose the following condition:
%\begin{align}\label{sa}
%\left\{
%\begin{aligned}
%&\lim_{t\to \infty}\left(\|u(t)-\sum_{k=1}^K\sigma_kQ(\cdot-z_k(t))\|_{H^1}+\|{\partial}_tu(t)\|_{L^2}\right)=0,
%\\
%&\lim_{t\to \infty}\min_{i\neq j}|z_i(t)-z_j(t)|=\infty,
%\\
%&\|u(t)-Q(\cdot-z(t))\|_{H^1}+\|{\partial}_tu(t)\|_{L^2}\ll 1,\ z(t)\gg 1\ \mbox{for} t\geq 0,
%\\
%&\int\{{\partial}_tu(t)+2\alpha(u(t)-\sum_{k=1}^K\sigma_kQ(\cdot-z_k(t)))\}{\partial}_lQ(\cdot-z_i(t))=0\ %\mbox{for}\ t\geq 0,\ 1\leq l\leq N,\ 1\leq i\leq K,
%\end{aligned}
%\r.
%\end{align}
%for some $z:[0,\infty)\to (\mathbb{R}^N)^K,\ \sigma \in\{-1,1\}^K$.

\subsection{Decomposition close to the sum of solitary waves}

In this subsection, we consider the global dynamics around a $K$-solitary wave. First, we define some notations. For $z:[0,\infty)\to (\mathbb{R}^d)^K$ and $\sigma\in \{-1,1\}^K$, we define
\begin{align*}
Q_k=\sigma_k Q(\cdot-z_k),\ \vec{Q}_k=
\begin{pmatrix}
Q_k\\
0   \\
\end{pmatrix}
,\ Q_{\sum}=\sum_{k=1}^KQ_k,\ \vec{Q}_{\sum}=
\begin{pmatrix}
Q_{\sum}
\\
0
\end{pmatrix}
, \phi_k=\phi(\cdot-z_k).
\end{align*}
Furthermore, for a solution $\vec{u}$ of \eqref{DNKG}, we define 
\begin{align*}
\vec{\varepsilon}=
\begin{pmatrix}
\varepsilon
\\
\eta
\end{pmatrix}
=\vec{u}-\vec{Q}_{\sum}.
\end{align*}
We note that $\vec{\varepsilon}$ depends only on $\vec{u}$, $z$, and $\sigma$. Furthermore, we define
\begin{align*}
D(t)=\min_{i\neq j}|z_i(t)-z_j(t)|.
\end{align*}
Here, to facilitate the subsequent discussion, we give the following conditions:
\begin{align}\label{sa}
\left\{
\begin{aligned}
&\lim_{t\to \infty}\|\vec{\varepsilon}(t)\|_{\mathcal{H}}=0,
\\
&\lim_{t\to \infty}D(t)=\infty,
\\
&\|\vec{\varepsilon}(t)\|_{\mathcal{H}}+\frac{1}{D(t)}<\delta\ \mbox{for}\ t\geq 0,
\\
&\int (\eta+2\alpha \varepsilon){\partial}_lQ(\cdot-z_k)=0\  \mbox{for}\ t\geq 0,\ 1\leq l\leq d,\ 1\leq i\leq K,
\end{aligned}
\r.
\end{align}
where $\delta>0$. Before we consider \eqref{DNKG}, we define
\begin{align*}
\mathcal{L}&=-\Delta+1-f^{\prime}(Q_{\sum}),
\\
N_1&=\sum_{k=1}^K \dot{z}_k\cdot \nabla Q_k,
\\
N_2&=f(u)-\left(\sum_{k=1}^Kf(Q_k)\right)-f^{\prime}(Q_{\sum})\varepsilon.
\end{align*}
Then, \eqref{DNKG} can be rewritten as
\begin{align}\label{epsilonsystem}
\frac{d}{dt}
\begin{pmatrix}
\varepsilon
\\
\eta
\end{pmatrix}
=
\begin{pmatrix}
0 & 1
\\
-\mathcal{L} & -2\alpha
\end{pmatrix}
\begin{pmatrix}
\varepsilon
\\
\eta
\end{pmatrix}
+
\begin{pmatrix}
N_1
\\
N_2
\end{pmatrix}
.
\end{align}

Here, we define 
\begin{align*}
R=\sum_{i\neq j}p|Q_i|^{p-1}Q_j.
\end{align*}
We will use $R$ later.

\subsection{Spectral decomposition of the solution}
In this subsection, we estimate $\|\vec{\varepsilon}\|_{\mathcal{H}}$. First, we define 
\begin{align*}
a_k^{\pm}&=\int ( \eta-\nu^{\mp}\varepsilon)\phi_k\ \mbox{for}\ 1\leq k\leq K,
\\
b_{k,l}&=\int \eta{\partial}_lQ_k\ \mbox{for}\ 1\leq k\leq K,\ 1\leq l\leq d.
\end{align*}

In addition, to succinctly represent the forces acting on the centers of the solitary waves, we define
\begin{align}\label{Fdef}
\mathcal{F}(r)=\frac{g(r)}{2\alpha\|{\partial}_1Q\|_{L^2}^2}.
\end{align}
Notably, $\mathcal{F}$ satisfies for $r>1$,
\begin{align}\label{Fesa}
|\mathcal{F}(r)-c_{\star}r^{-\frac{d-1}{2}}e^{-r}|+|\mathcal{F}^{\prime}(r)+c_{\star}r^{-\frac{d-1}{2}}e^{-r}|\lesssim r^{-\frac{d+1}{2}}e^{-r},
\end{align}
where $c_{\star}=\frac{c_g}{2\alpha \|{\partial}_1Q\|_{L^2}^2}$.\ In particular, we have for $1\ll r<r^{\prime}$,
\begin{align}\label{Fesb}
\mathcal{F}(r^{\prime})(r^{\prime}-r)\lesssim\mathcal{F}(r)-\mathcal{F}(r^{\prime})\lesssim \mathcal{F}(r)(r^{\prime}-r).
\end{align}
\begin{lemma}\label{expansionlem}
Let $\vec{u}$ be a solution of \eqref{DNKG} satisfying \eqref{sa} for some $z,\ \sigma$, and $0<\delta\ll 1$.\ Then for any $1<\theta<\theta_{\star}$, we have for $1\leq k\leq K$ and $1\leq l\leq d$,
\begin{align}
\left|\dot{z}_k+\sum_{1\leq i\leq K,\ i\neq k}\sigma_i\sigma_k\mathcal{F}(|z_k-z_i|) \frac{z_k-z_i}{|z_k-z_i|}\right|&\lesssim e^{-\theta D}+\|\vec{\varepsilon}\|_{\mathcal{H}}^2,\label{zes}
\\
\left| \frac{d}{dt}a_k^{\pm}-\nu^{\pm}a_k^{\pm}\right|&\lesssim \mathcal{F}(D)+\|\vec{\varepsilon}\|_{\mathcal{H}}^2,\label{apmes}
\\
\left|\frac{d}{dt}b_{k,l}+2\alpha b_{k,l}\right|&\lesssim \mathcal{F}(D)+\|\vec{\varepsilon}\|_{\mathcal{H}}^2.\label{bes}
\end{align}
\end{lemma}
\begin{proof}
This proof rewrites \cite[Section3, Section4]{IN} in simplified terms. Differentiating the equation $\int (\eta+2\alpha\varepsilon){\partial}_lQ(\cdot-z_k)=0$, we obtain
\begin{align*}
0&=\int({\partial}_t\eta+2\alpha {\partial}_t\varepsilon){\partial}_lQ(\cdot-z_k)-\int(\eta+2\alpha\varepsilon)\dot{z}_k\cdot \nabla({\partial}_lQ_k).
\end{align*}
In particular, we have
\begin{align*}
\int ({\partial}_t\eta+2\alpha {\partial}_t\varepsilon){\partial}_lQ(\cdot-z_k)&=\int (-\mathcal{L}\varepsilon+2\alpha N_1+N_2){\partial}_lQ(\cdot-z_k).
\end{align*}
Here, we estimate the right-hand side. First, we have
\begin{align*}
\int-\mathcal{L}\varepsilon{\partial}_lQ(\cdot-z_k)&=\int \varepsilon (f^{\prime}(Q_{\sum})-f^{\prime}(Q_k)){\partial}_lQ(\cdot-z_k). 
\end{align*}
By considering the Taylor expansion, we have
\begin{align*}
\left| (f^{\prime}(Q_{\sum})-f^{\prime}(Q_k))({\partial}_lQ(\cdot-z_k))\right|\lesssim \sum_{1\leq i\leq K,\ i\neq k}(|Q_k|^{p-1}|Q_i|+|Q_i|^{p-1}|Q_k|).
\end{align*}
Therefore we have
\begin{align}\label{zespart1}
\left|\int \varepsilon (f^{\prime}(Q_{\sum})-f^{\prime}(Q_k)){\partial}_lQ(\cdot-z_k)\right|\lesssim \|\vec{\varepsilon}\|_{\mathcal{H}}\mathcal{F}(D)\lesssim \|\vec{\varepsilon}\|_{\mathcal{H}}^2+\mathcal{F}(D)^2.
\end{align}
Second, by direct computation we obtain
\begin{align}\label{zespart2}
\begin{aligned}
\int 2\alpha N_1{\partial}_lQ(\cdot-z_k)&=2\alpha \sum_{i=1}^K \int  (\dot{z}_i\cdot \nabla Q_i){\partial}_lQ(\cdot-z_k)
\\
&=2\alpha \sigma_k \dot{z}_{k,l}\|{\partial}_1Q\|_{L^2}^2+O\left(e^{-\frac{3}{4}D} \left(\sum_{i=1}^K|\dot{z}_i|\right)\right).
\end{aligned}
\end{align}
Note that we obtain \eqref{zespart2} by using \eqref{Qsymmetry} and Lemma \ref{ni}. Next, we estimate $\int N_2{\partial}_lQ(\cdot-z_k)$. $N_2$ sarisfies
\begin{align*}
N_2-R=(f(u)-f(Q_{\sum})-f^{\prime}(Q_{\sum})\varepsilon)+\left(f(Q_{\sum})-\sum_{i=1}^Kf(Q_i)-R\right).
\end{align*}
Then we obtain
\begin{align*}
\left|f(u)-f(Q_{\sum})-f^{\prime}(Q_{\sum})\varepsilon\right|\lesssim |\varepsilon|^2,
\end{align*}
and
\begin{align*}
|f(Q_{\sum})-\sum_{i=1}^Kf(Q_i)-R|\lesssim e^{-\theta D}.
\end{align*}
Therefore we have
\begin{align}\label{zespart3}
\left|\int (N_2-R){\partial}_lQ(\cdot-z_k)\right|\lesssim e^{-\theta D}+\|\vec{\varepsilon}\|_{\mathcal{H}}^2.
\end{align}
In addition, we have
\begin{align}\label{zespart4}
\begin{aligned}
\int R{\partial}_lQ(\cdot-z_k)&=\sum_{i=1}^K\int p|Q_k|^{p-1}{\partial}_lQ(\cdot-z_k)Q_i+O(e^{-\theta D})
\\
&=\sum_{1\leq i\leq K,\ i\neq k}2\alpha\sigma_i \|{\partial}_1Q\|_{L^2}^2\mathcal{F}(|z_k-z_i|) \frac{z_{k,l}-z_{i,l}}{|z_k-z_i|}+O(e^{-\theta D}).
\end{aligned}
\end{align}
By \eqref{zespart1}-\eqref{zespart4}, we have
\begin{align*}
0&=2\alpha\|{\partial}_1Q\|_{L^2}^2\left( \sigma_k\dot{z}_{k,l}+\sum_{1\leq i\leq K,\ i\neq k}\sigma_i \mathcal{F}(|z_k-z_i|) \frac{z_{k,l}-z_{i,l}}{|z_k-z_i|}\right)
\\
&\quad +O\left(e^{-\theta D}+\|\vec{\varepsilon}\|_{\mathcal{H}}^2+\left(\sum_{i=1}^K|\dot{z}_i|\right)\|\vec{\varepsilon}\|_{\mathcal{H}}\right).
\end{align*}
In particular, $z$ satisfies
\begin{align*}
\sum_{i=1}^K |\dot{z}_i|\lesssim \mathcal{F}(D)+\|\vec{\varepsilon}\|_{\mathcal{H}}^2,
\end{align*}
and therefore we obtain \eqref{zes}.

Next, we estimate \eqref{apmes} and \eqref{bes}. By Lemma \ref{ni}, we have
\begin{align*}
\frac{d}{dt}a_k^{\pm}&=\int ({\partial}_t\eta-\nu^{\mp}{\partial}_t\varepsilon)\phi_k+O(\mathcal{F}(D)+\|\vec{\varepsilon}\|_{\mathcal{H}}^2)
\\
&=\int(-\mathcal{L}\varepsilon-2\alpha \eta-\nu^{\mp}\eta)\phi_k+O(\mathcal{F}(D)+\|\vec{\varepsilon}\|_{\mathcal{H}}^2)
\\
&=\nu^{\pm}a_k^{\pm}+\int (f^{\prime}(Q_{\sum})-f^{\prime}(Q_k))\phi_k+O(\mathcal{F}(D)+\|\vec{\varepsilon}\|_{\mathcal{H}}^2)
\\
&=\nu^{\pm}a_k^{\pm}+O(\mathcal{F}(D)+\|\vec{\varepsilon}\|_{\mathcal{H}}^2),
\end{align*}
and
\begin{align*}
\frac{d}{dt}b_{k,l}&=\int {\partial}_t\eta {\partial}_lQ_k+O(\mathcal{F}(D)+\|\vec{\varepsilon}\|_{\mathcal{H}}^2)
\\
&=-2\alpha b_{k,l}+\int (f^{\prime}(Q_{\sum})-f^{\prime}(Q_k)){\partial}_lQ_k+O(\mathcal{F}(D)+\|\vec{\varepsilon}\|_{\mathcal{H}}^2)
\\
&=-2\alpha b_{k,l}+O(\mathcal{F}(D)+\|\vec{\varepsilon}\|_{\mathcal{H}}^2).
\end{align*}
Hence we complete the proof.
\end{proof}
Since $\int (\eta+2\alpha\varepsilon){\partial}_lQ(\cdot-z_k)=0$ holds, we have for $1\leq k\leq K$ and $1\leq l\leq d$,
\begin{align*}
\langle \varepsilon, \phi_k \rangle=\frac{a_k^+-a_k^-}{\nu^+-\nu^-},\ \langle \varepsilon,{\partial}_lQ_k\rangle= -\frac{b_{k,l}}{2\alpha}.
\end{align*}
Therefore, we obtain
\begin{align}\label{coersim}
\sum_{k=1}^K\left(|\langle \varepsilon,\phi_k\rangle|+\left(\sum_{l=1}^d|\langle \varepsilon,{\partial}_lQ_k\rangle|\right) \right)\sim \sum_{k=1}^K\left( |a_k^+|+|a_k^-|+\left(\sum_{l=1}^d|b_{k,l}|\right)\right).
\end{align}

\subsection{Energy estimates}

In this subsection, we assume that $\vec{u}$ is a solution of \eqref{DNKG} satisfying \eqref{sa} for some $\sigma,\ z$, and $0<\delta\ll 1$. For $\mu>0$ sufficiently small to be chosen, we denote $\rho=2\alpha-\mu$. Furthermore, we define $\mathcal{E}$ as
\begin{align*}
\mathcal{E}=\frac{1}{2}\int \left(|\nabla\varepsilon|^2+(1-\rho \mu)\varepsilon^2+(\eta+\mu \varepsilon)^2-2(F(u)-F(Q_{\sum})-f(Q_{\sum})\varepsilon)\right),
\end{align*}
where $F(u)=\frac{1}{p+1}|u|^{p+1}$. First, we introduce the coercivity of $\mathcal{E}$.
\begin{lemma}\label{enelem}
There exist $\mu>0$,\ $\tilde{C}_1>0,\ \tilde{C}_2>0,\ \tilde{C}_3>0$, and $\delta_1>0$ such that the following holds. Let $\vec{u}$ be a solution of \eqref{DNKG} satisfying \eqref{sa} for some $\sigma,\ z$, and $0<\delta<\delta_1$. Then, we have
 \begin{align}\label{Esiml}
\tilde{C}_2\|\vec{\varepsilon}\|_{\mathcal{H}}^2\leq \mathcal{E}+\tilde{C}_1\sum_{k=1}^K\left( |a_k^+|^2+|a_k^-|^2+(\sum_{l=1}^d|b_{k,l}|^2)\right)\leq\frac{1}{\tilde{C}_2}\|\vec{\varepsilon}\|_{\mathcal{H}}^2,
\end{align}
and 
\begin{align}\label{Edecay}
\frac{d}{dt}\mathcal{E}+2\mu \mathcal{E}\leq \tilde{C}_3\left( \|\vec{\varepsilon}\|_{\mathcal{H}}^3+\|\vec{\varepsilon}\|_{\mathcal{H}}\mathcal{F}(D)\right).
\end{align}
\end{lemma}
\begin{remark}
$\mu>0$,\ $\tilde{C}_1>0,\ \tilde{C}_2>0,\ \tilde{C}_3>0$, and $\delta_1>0$ only depend on $d,\ \alpha,\ p$.
\end{remark}

\begin{proof}
First, we estimate $\langle \mathcal{L}\varepsilon,\varepsilon\rangle$. Let $\chi$ be a smooth function satisfying the following properties:
\begin{align*}
\chi=1\ \mbox{on}\ [0,1],\ \ \chi=0\ \mbox{on}\ [2,\infty),\ \ \chi^{\prime}\leq 0 \ \mbox{on}\ \mathbb{R}.
\end{align*}
For $\lambda=\frac{D}{4}\gg 1$, let
\begin{align*}
\chi_k(x)=\chi\left( \frac{|x-z_k|}{\lambda}\right),\ \chi_c(x)=\left(1-\sum_{k=1}^K\chi_k(x)^2\right)^{\frac{1}{2}}
\end{align*}
for $1\leq k\leq K$. Furthermore, we define
\begin{align}
\varepsilon_k=\varepsilon \chi_k,\ \ \varepsilon_c=\varepsilon\chi_c\label{cutlem1}
\end{align}
for $1\leq k\leq K$. Then, we have
\begin{align}
\varepsilon^2=\left(\sum_{k=1}^K \varepsilon_k^2\right)+\varepsilon_c^2.
\end{align}
In addition, we have
\begin{align}
\|\varepsilon\|_{H^1}^2=\left(\sum_{k=1}^K\|\varepsilon_k\|_{H^1}^2\right)+\|\varepsilon_c\|_{H^1}^2+O\left(\frac{\|\varepsilon\|_{H^1}^2}{D}\right).
\end{align}
We also have for $1\leq k\leq K$ and $1\leq l\leq d$,
\begin{align}
\left| \langle \varepsilon-\varepsilon_k,{\partial}_lQ_k\rangle\right|+\left|\langle \varepsilon-\varepsilon_k,\phi_k\rangle\right|\lesssim e^{-\frac{D}{4}}\|\varepsilon\|_{L^2},
\\
\left|\langle f^{\prime}(Q_{\sum})-f^{\prime}(Q_k),\varepsilon_k^2\rangle\right|\lesssim e^{-\theta D}\|\varepsilon\|_{L^2},
\\
\langle \mathcal{L}\varepsilon_c,\varepsilon_c\rangle=\|\varepsilon_c\|_{L^2}^2+O(e^{-\frac{D}{4}}\|\varepsilon\|_{L^2}),\label{cutlem2}
\end{align}
where $0<\theta<\min{(p-2,\frac{1}{4})}$. By \eqref{cutlem1}-\eqref{cutlem2}, we obtain
\begin{align}\label{cutes1}
\begin{aligned}
\langle \mathcal{L}\varepsilon,\varepsilon\rangle&=\left(\sum_{k=1}^K \langle \mathcal{L}\varepsilon_k,\varepsilon_k\rangle\right)+\langle \mathcal{L}\varepsilon_c,\varepsilon_c\rangle+O\left(\frac{\|\varepsilon\|_{H^1}^2}{D}\right)
\\
&=\left(\sum_{k=1}^K\langle (-\Delta+1-f^{\prime}(Q_k))\varepsilon_k,\varepsilon_k\rangle \right)+\|\varepsilon_c\|_{L^2}^2+O\left(\frac{\|\varepsilon\|_{H^1}^2}{D}\right),
\end{aligned}
\end{align}
and
\begin{align}\label{cutes2}
\begin{aligned}
\sum_{k=1}^K\left(|\langle \varepsilon,\phi_k\rangle|+\left(\sum_{l=1}^d|\langle \varepsilon,{\partial}_lQ_k\rangle|\right) \right)&=\sum_{k=1}^K\left(|\langle \varepsilon_k,\phi_k\rangle|+\left(\sum_{l=1}^d|\langle \varepsilon_k,{\partial}_lQ_k\rangle|\right) \right)
\\
&\quad +O(e^{-\frac{D}{4}}\|\varepsilon\|_{L^2}).
\end{aligned}
\end{align}
By Lemma \ref{Lproperty}, we have for $1\leq k\leq K$,
\begin{align}\label{cutes3}
\langle (-\Delta+1-f^{\prime}(Q_k))\varepsilon_k,\varepsilon_k\rangle\geq c_L\|\varepsilon_k\|_{L^2}^2-\frac{1}{c_L}\left(\langle \varepsilon_k,\phi_k\rangle^2+\sum_{l=1}^d\langle\varepsilon_k,{\partial}_lQ_k\rangle^2\right).
\end{align}
By \eqref{coersim} and \eqref{cutes1}-\eqref{cutes3}, there exists $\tilde{\delta}>0$ such that the following holds: when $\vec{u}$ satisfies \eqref{sa} for some $z,\ \sigma$, and $0<\delta<\tilde{\delta}$, there exists $\hat{C}>0$ depending only on $d,\ \alpha,\ p$ such that 
\begin{align}\label{Lsumsim}
\|\varepsilon\|_{H^1}^2\sim \langle \mathcal{L}\varepsilon,\varepsilon\rangle+\hat{C} \sum_{k=1}^K\left( |a_k^+|^2+|a_k^-|^2+\left(\sum_{l=1}^d|b_{k,l}|^2\right)\right).
\end{align}
Next, by direct computation, we have
\begin{align}\label{Eespart1}
\left|\int(F(u)-F(Q_{\sum})-f(Q_{\sum})\varepsilon-\frac{1}{2}f^{\prime}(Q_{\sum})\varepsilon^2)\right|\lesssim \|\vec{\varepsilon}\|_{\mathcal{H}}^3.
\end{align}
Furthermore, we have
\begin{align}\label{Eespart2}
\begin{aligned}
&\int  (|\nabla\varepsilon|^2+(1-\rho \mu)\varepsilon^2+(\eta+\mu \varepsilon)^2-f^{\prime}(Q_{\sum})\varepsilon^2)
\\
&\quad=\langle \mathcal{L}\varepsilon,\varepsilon\rangle +\|\eta\|_{L^2}^2+2\mu \langle \varepsilon,\eta\rangle -\mu(\rho-\mu)\|\varepsilon\|_{L^2}^2.
\end{aligned}
\end{align}
Therefore by \eqref{Lsumsim}, \eqref{Eespart1}, \eqref{Eespart2}, we obtain \eqref{Esiml} if $\mu>0$ and $\tilde{\delta}>0$ are sufficiently small.

Next, we estimate \eqref{Edecay}. This proof is due to \cite[Lemma 2.6]{CMYZ}. In this paper, the choice of the center is different from \cite{CMYZ}. Therefore, we give the proof.

Differentiating $\mathcal{E}$, we have
\begin{align*}
\frac{d}{dt}\mathcal{E}=\int \left({\partial}_t\varepsilon(\mathcal{L}\varepsilon-\rho \mu \varepsilon)+({\partial}_t\eta+\mu {\partial}_t\varepsilon)(\eta+\mu\varepsilon)\right)+O(\|\vec{\varepsilon}\|_{\mathcal{H}}(\|\vec{\varepsilon}\|_{\mathcal{H}}^2+\mathcal{F}(D))).
%\\
%&\quad-\int {\partial}_t\varepsilon(f(u)-f(Q_{\sum})-f^{\prime}(Q_{\sum})\varepsilon)+O(\|\vec{\varepsilon}\|_{\mathcal{H}}(\|\vec{\varepsilon}\|_{\mathcal{H}}^2+\mathcal{F}(D))).
\end{align*}
In particular, by \eqref{Eespart1} and \eqref{Eespart2}, we have
\begin{align*}
\left|\mathcal{E}-\frac{1}{2}(\langle \mathcal{L}\varepsilon,\varepsilon\rangle +\|\eta\|_{L^2}^2+2\mu \langle \varepsilon,\eta\rangle -\mu(\rho-\mu)\|\varepsilon\|_{L^2}^2)\right|\lesssim \|\vec{\varepsilon}\|_{\mathcal{H}}^3,
\end{align*}
and
\begin{align*}
\int ({\partial}_t\eta+\mu {\partial}_t\varepsilon)(\eta+\mu\varepsilon)&=-\int \mathcal{L}\varepsilon({\partial}_t\varepsilon+\mu \varepsilon)+(2\alpha-\mu)\eta(\eta+\mu\varepsilon)
\\
&\quad+O(\|\vec{\varepsilon}\|_{\mathcal{H}}(\|\vec{\varepsilon}\|_{\mathcal{H}}^2+\mathcal{F}(D))).
\end{align*}
Gathering these arguments, we have
\begin{align*}
\frac{d}{dt}\mathcal{E}=-2\mu \mathcal{E}-2(\alpha-\mu)\|\eta+\mu \varepsilon\|_{L^2}^2+O(\|\vec{\varepsilon}\|_{\mathcal{H}}(\|\vec{\varepsilon}\|_{\mathcal{H}}^2+\mathcal{F}(D))),
\end{align*}
which implies \eqref{Edecay}.
\end{proof}
In the following discussion, we choose $\mu>0$ as Lemma \ref{enelem}. Here, we consider the Taylor expansion of the Hamiltonian. We define $\theta_{\ast}>0$ as a constant satisfying
\begin{align}\label{thetaast}
\theta_{\ast}=\frac{4}{3-\theta_{\star}}>2.
\end{align}
\begin{lemma}\label{Hamilem}
Let $\vec{u}$ be a solution of \eqref{DNKG} satisfying \eqref{sa} for some $\sigma,\ z$, and $0<\delta<\delta_1$. Then there exists $c_{\ast}>0$ such that
\begin{align}\label{Hamieq}
\begin{aligned}
E(\vec{u})-KE(Q,0)&=-c_{\ast}\sum_{1\leq i<j\leq K}\sigma_i\sigma_j\mathcal{F}(|z_i-z_j|)+\frac{1}{2}(\langle \mathcal{L}\varepsilon,\varepsilon\rangle+\|\eta\|_{L^2}^2)
\\
&\quad+O(D^{-1}\mathcal{F}(D)+\|\vec{\varepsilon}\|_{\mathcal{H}}^{\theta_{\ast}}).
\end{aligned}
\end{align}
Furthermore, there exists $C_E>0$ such that
\begin{align}\label{Hamilower}
\mathcal{E}\leq C_E\left(E(\vec{u})-KE(Q,0)+\sum_{k=1}^K\left( |a_k^+|^2+|a_k^-|^2+(\sum_{l=1}^d|b_{k,l}|^2)\right)+\mathcal{F}(D)\right).
\end{align}
\end{lemma}
\begin{proof}
By direct computation, we have
\begin{align*}
E(\vec{u})-KE(Q,0)&=\frac{1}{2}(\langle \mathcal{L}\varepsilon,\varepsilon\rangle+\|\eta\|_{L^2}^2)
\\
&\quad-\frac{1}{p+1}\int \left(|u|^{p+1}-|Q_{\sum}|^{p+1}-(p+1)f(Q_{\sum})\varepsilon-\frac{p+1}{2}f^{\prime}(Q_{\sum})\varepsilon^2\right)
\\
&\quad -\frac{1}{p+1}\int\left( |Q_{\sum}|^{p+1}-\sum_{k=1}^K |Q_k|^{p+1}\right)
\\
&\quad+\frac{1}{2}\int \left( f^{\prime}(Q_{\sum})-\sum_{k=1}^Kf^{\prime}(Q_k) \right)\varepsilon^2.
\end{align*}
By the Taylor expansion, we obtain
\begin{align*}
||u|^{p+1}-|Q_{\sum}|^{p+1}-(p+1)f(Q_{\sum})\varepsilon-\frac{p+1}{2}f^{\prime}(Q_{\sum})\varepsilon^2|&\lesssim |\varepsilon|^3,
\\
\left|\left( f^{\prime}(Q_{\sum})-\sum_{k=1}^Kf^{\prime}(Q_k) \right)\varepsilon^2\right|&\lesssim e^{-(\theta_{\star}-1) D}|\varepsilon|^2
\\
&\lesssim e^{-2D}+\|\vec{\varepsilon}\|_{\mathcal{H}}^{\theta_{\ast}}
\\
&\lesssim D^{-1}\mathcal{F}(D)+\|\vec{\varepsilon}\|_{\mathcal{H}}^{\theta_{\ast}}.
\end{align*}
Furthermore, we have
\begin{align*}
-\frac{1}{p+1}\int\left( |Q_{\sum}|^{p+1}-\sum_{k=1}^K |Q_k|^{p+1}\right)=-2\sum_{1\leq i<j\leq K}\sigma_i\sigma_j\int |Q_k|^p|Q_l|+O\left(D^{-1}\mathcal{F}(D) \right).
\end{align*}
Notably, by the definition of $\mathcal{F}$, there exists $c_1>0$ such that 
\begin{align*}
\left|\int |Q_i|^p|Q_j|-c_1\mathcal{F}(|z_i-z_j|)\right|\lesssim D^{-1}\mathcal{F}(D).
\end{align*}
Therefore we obtain \eqref{Hamieq} by choosing $c_{\ast}=2c_1$. \eqref{Hamilower} follows from \eqref{Lsumsim} and \eqref{Hamieq}, since $\vec{u}$ satisfies $E(\vec{u})-KE(Q,0)\geq 0$.
\end{proof}

\subsection{Estimate of the instability term}

Here, we estimate the instability term. Let $L>\tilde{C}_1$ be large to be chosen. We define $\mathcal{G}$ as
\begin{align}\label{Gdef}
\mathcal{G}=\mathcal{E}+L\sum_{k=1}^K\left(|a_k^-|^2+\left(\sum_{l=1}^d|b_{k,l}|^2\right)\right).
\end{align}
Furthermore, we define
\begin{align*}
|a^+|^2=\sum_{k=1}^K |a_k^+|^2.
\end{align*}
Then, by Lemma \ref{enelem}, when $\vec{u}$ satisfies \eqref{sa} for some $z,\ \sigma$, and $0<\delta<\delta_1$, we have
\begin{align}\label{Gdouchi}
\mathcal{G}+L|a^+|^2\sim \|\vec{\varepsilon}\|_{\mathcal{H}}^2. 
\end{align}
Here, we estimate $\mathcal{G}$.

\begin{lemma}\label{Glem}
There exist $L>0$ and $0<\delta_2<\delta_1$ and $\tau>0$ such that the following holds. Let $\vec{u}$ be a solution of \eqref{DNKG} satisfying \eqref{sa} for some $\sigma,\ z$, and $0<\delta<\delta_2$. Then we have for all $t\geq0$
\begin{align}\label{Ges1}
\frac{d}{dt}\mathcal{G}\leq -\tau\cdot \mathcal{G}+\frac{1}{\tau}\left(|a^+|^2+\mathcal{F}(D)^2\right).
\end{align}
\end{lemma}

\begin{remark}
In Lemma \ref{Glem}, $L>0$ and $\delta_2>0$ and $\tau>0$ depend only on $d,\ p,\ \alpha$.

\end{remark}

\begin{proof}
First, by Lemma \ref{expansionlem}, we have for $1\leq k\leq K$,\ $1\leq l\leq d$, 
\begin{align}\label{abes}
\begin{aligned}
\left| \frac{d}{dt}(a_k^-)^2-2\nu^-(a_k^-)^2 \right|&\lesssim \|\vec{\varepsilon}\|_{\mathcal{H}}^3+\mathcal{F}(D)\|\vec{\varepsilon}\|_{\mathcal{H}},
\\
\left| \frac{d}{dt}(b_{k,l})^2+4\alpha (b_{k,l})^2\right|&\lesssim \|\vec{\varepsilon}\|_{\mathcal{H}}^3+\mathcal{F}(D)\|\vec{\varepsilon}\|_{\mathcal{H}}.
\end{aligned}
\end{align}
By \eqref{Edecay} and \eqref{abes}, we have
\begin{align}\label{Glesssim}
\frac{d}{dt}\mathcal{G}+2\mu \mathcal{E}+2L\sum_{k=1}^K\left( -\nu^-|a_k^-|^2+(\sum_{l=1}^d 2\alpha |b_{k,l}|^2)\right)\lesssim \|\vec{\varepsilon}\|_{\mathcal{H}}^3+\mathcal{F}(D)\|\vec{\varepsilon}\|_{\mathcal{H}}.
\end{align}
Here, we choose $L>0$ and $\epsilon_1>0$ satisfying 
\begin{align}\label{epsilon1lower}
2\mu \mathcal{E}+2L\sum_{k=1}^K\left( -\nu^-|a_k^-|^2+(\sum_{l=1}^d 2\alpha |b_{k,l}|^2)\right)+\epsilon_1^{-1}|a^+|^2\geq \epsilon_1(\mathcal{G}+\|\vec{\varepsilon}\|_{\mathcal{H}}^2).
\end{align}
Notably, \eqref{epsilon1lower} holds true by \eqref{Esiml} and \eqref{Gdouchi}. We note that $L>0$ and $\epsilon_1>0$ depend only on $d,\ p,\ \alpha$. By \eqref{Glesssim} and \eqref{epsilon1lower}, there exists $C_1>0$ such that 
\begin{align}\label{Gespart1}
\frac{d}{dt}\mathcal{G}+\epsilon_1\mathcal{G}+\epsilon_1\|\vec{\varepsilon}\|_{\mathcal{H}}^2-\epsilon_1^{-1}|a^+|^2\leq C_1\|\vec{\varepsilon}\|_{\mathcal{H}}^3+C_1\mathcal{F}(D)\|\vec{\varepsilon}\|_{\mathcal{H}}.
\end{align}
 By the smallness of $\delta_2>0$ and the arithmetic-geometric mean inequality, we have
\begin{align}\label{Gespart2}
C_1\|\vec{\varepsilon}\|_{\mathcal{H}}^3+C_1\mathcal{F}(D)\|\vec{\varepsilon}\|_{\mathcal{H}}\leq \frac{1}{2}\epsilon_1\|\vec{\varepsilon}\|_{\mathcal{H}}^2+\left(\frac{\epsilon_1}{2}\|\vec{\varepsilon}\|_{\mathcal{H}}^2+\frac{C_1^2}{2\epsilon_1}\mathcal{F}(D)^2\right).
\end{align}
By \eqref{Gespart1} and \eqref{Gespart2}, we obtain
\begin{align*}
\frac{d}{dt}\mathcal{G}+\epsilon_1\mathcal{G}\leq \frac{1}{\epsilon_1}|a^+|^2+\frac{C_1^2}{2\epsilon_1}\mathcal{F}(D)^2.
\end{align*}
Therefore by choosing
\begin{align*}
\tau=\min{\left(\epsilon_1, \frac{2\epsilon_1}{C_1^2}\right)},
\end{align*}
we obtain \eqref{Ges1}. It is clear that $\tau>0$ and $\delta_2>0$ depend only on $d,\ p,\ \alpha$, and we complete the proof. 
\end{proof}

Now, we estimate $|a^+|$. 
\begin{lemma}\label{a+lem1}
Let $\vec{u}$ be a solution of \eqref{DNKG} satisfying \eqref{sa} for some $\sigma,\ z$, and $\delta>0$. Then we have
\begin{align}\label{lima}
\lim_{t\to\infty} \frac{|a^+(t)|^2}{\|\vec{\varepsilon}(t)\|_{\mathcal{H}}^2+\mathcal{F}(D(t))}=0.
\end{align}
\end{lemma}
\begin{proof}
First, by Lemma \ref{enelem} and Lemma \ref{Hamilem}, there exists $C_1>0$ such that 
\begin{align}\label{a+espart0}
\begin{aligned}
\|\vec{\varepsilon}\|_{\mathcal{H}}^2+\mathcal{F}(D)\leq C_1\left(E(\vec{u})-KE(Q,0)+|a^+|^2+\sum_{k=1}^K(|a_k^-|^2+(\sum_{l=1}^d|b_{k,l}|^2))+\mathcal{F}(D)\right),
\\
E(\vec{u})-KE(Q,0)+|a^+|^2+\sum_{k=1}^K(|a_k^-|^2+(\sum_{l=1}^d|b_{k,l}|^2))+\mathcal{F}(D)<C_1\left(\|\vec{\varepsilon}\|_{\mathcal{H}}^2+\mathcal{F}(D)\right).
\end{aligned}
\end{align}
We fix $M>0$. Let $\delta_M>0$ be sufficiently small. By \eqref{sa}, there exists $T_M>0$ such that for $t\geq T_M$
\begin{align}\label{Massu}
\frac{1}{D(t)}+\mathcal{F}(D(t))+\|\vec{\varepsilon}(t)\|_{\mathcal{H}}+|a^+(t)|<\delta_M<\delta_2.
\end{align}
Here, we assume that $T_1>T_M$ satisfies 
\begin{align*}
\frac{1}{M}\left( \|\vec{\varepsilon}(T_1)\|_{\mathcal{H}}^2+\mathcal{F}(D(T_1))\right)\leq |a^+(T_1)|^2.
\end{align*}
Then, we define 
\begin{align}\label{M1def}
M_1=\max{\left(M, C_1^2M, \frac{C_1(2+\nu^+)}{\nu^+} \right)}+1,
\end{align}
and we introduce the following bootstrap estimate
\begin{align}\label{bses1}
\frac{1}{M_1}\left( \|\vec{\varepsilon}\|_{\mathcal{H}}^2+\mathcal{F}(D)\right)\leq |a^+|^2\leq \delta_M^2.
\end{align}
We define $T_2>T_1$ as 
\begin{align*}
T_2=\sup\{t\in[T_1,\infty)\ \mbox{such that \eqref{bses1} holds on}\ [T_1,t]\}.
\end{align*}
Since $\delta_M>0$ is sufficiently small, by Lemma \ref{expansionlem} we have for $T_1<t<T_2$
\begin{align}\label{a+espart1}
\begin{aligned}
\frac{d}{dt}|a^+|^2\geq \nu^+|a^+|^2,
\\
\frac{d}{dt}\left(\sum_{k=1}^K(|a_k^-|^2+(\sum_{l=1}^d|b_{k,l}|^2)\right)\leq |a^+|^2.
\end{aligned}
\end{align}
Furthermore, $D$ satisfies
\begin{align*}
\left| \frac{d}{dt}D\right|\lesssim \|\vec{\varepsilon}\|_{\mathcal{H}}^2+\mathcal{F}(D).
\end{align*}
Therefore, we have
\begin{align}\label{a+espart2}
\left| \frac{d}{dt}\mathcal{F}(D)\right|\leq |a^+|^2.
\end{align}
Here, we define
\begin{align*}
\mathcal{I}=M_1|a^+|^2- C_1\left(E(\vec{u})-KE(Q,0)+|a^+|^2+\sum_{k=1}^K(|a_k^-|^2+(\sum_{l=1}^d|b_{k,l}|^2))+\mathcal{F}(D)\right).
\end{align*}
Then we have $\mathcal{I}(T_1)>0$. Furthermore, by \eqref{energydecay}, \eqref{M1def}, \eqref{a+espart1}, and \eqref{a+espart2}, we have for $T_1<t<T_2$
\begin{align*}
\frac{d}{dt}\mathcal{I}\geq \left( (M_1-C_1)\nu^+-2C_1\right)|a^+|^2\geq 0.
\end{align*}
Therefore as long as \eqref{bses1} holds, we have
\begin{align*}
C_1\left(E(\vec{u})-KE(Q,0)+|a^+|^2+\sum_{k=1}^K(|a_k^-|^2+(\sum_{l=1}^d|b_{k,l}|^2))+\mathcal{F}(D)\right)<M_1|a^+|^2.
\end{align*} 
In addition, by \eqref{a+espart0} we have for $T_1<t<T_2$
\begin{align}\label{bses2}
\|\vec{\varepsilon}\|_{\mathcal{H}}^2+\mathcal{F}(D)<M_1|a^+|^2.
\end{align}
This implies that as long as \eqref{bses1} holds, we have \eqref{bses2}. Therefore for $T_1\leq t\leq T_2$, \eqref{bses2} holds, and $|a^+|^2$ increases exponentially. Thus $|a^+(T_3)|=\delta_M$ holds for some $T_3>T_1$, which contradicts \eqref{Massu}. Therefore for all $t\geq T_M$,
\begin{align*}
|a^+|^2\leq \frac{1}{M}\left( \|\vec{\varepsilon}\|_{\mathcal{H}}^2+\mathcal{F}(D)\right)
\end{align*}
holds, which implies \eqref{lima}.
\end{proof}

By Lemma \ref{a+lem1}, we estimate $\|\vec{\varepsilon}\|_{\mathcal{H}}$ and derive the conditions for nonlinear interactions. We define $V$ as 
\begin{align}\label{Vdef}
V=-\sum_{1\leq i<j\leq K}\sigma_i\sigma_j\mathcal{F}(|z_i-z_j|).
\end{align}
\begin{lemma}\label{limeplem}
Let $\vec{u}$ be a solution of \eqref{DNKG} satisfying \eqref{sa} for some $\sigma,\ z$, and $\delta>0$. Then we have
\begin{align}\label{limep}
\lim_{t\to\infty} \frac{\|\vec{\varepsilon}(t)\|_{\mathcal{H}}^2}{\mathcal{F}(D(t))}=0.
\end{align}
Furthermore, we have
\begin{align}\label{Ves1}
\liminf_{t\to\infty} \frac{V(t)}{\mathcal{F}(D(t))}\geq 0.
\end{align}
\end{lemma}
\begin{proof}
Since \eqref{Gdouchi} holds true, we only need to show
\begin{align}\label{Glim}
\lim_{t\to\infty} \frac{\mathcal{G}(t)}{\mathcal{F}(D(t))}=0.
\end{align}
We define 
\begin{align*}
S(t)=\frac{\mathcal{G}(t)}{\mathcal{F}(D(t))}.
\end{align*}
We fix $\epsilon>0$. We may assume $0<\delta<\delta_2$. By \eqref{Gdouchi} and \eqref{lima}, there exists $C_1>0$ such that 
\begin{align}\label{C1es}
\mathcal{G}+\epsilon\mathcal{F}(D)\geq C_1\|\vec{\varepsilon}\|_{\mathcal{H}}^2.
\end{align}
We note that $C_1$ does not depend on $\epsilon$. Then, by \eqref{Fesa},\eqref{Ges1}, and \eqref{C1es}, there exists $T_{\epsilon}>0$  such that for $t>T_{\epsilon}$, 
\begin{align*}
S^{\prime}&\leq \frac{ -\tau \mathcal{G}+\frac{1}{\tau}\left(|a^+|^2+\mathcal{F}(D)^2\right)}{\mathcal{F}(D)}-\frac{\mathcal{G}\mathcal{F}^{\prime}(D)\frac{d}{dt}D}{\mathcal{F}(D)^2}
\\
&\leq \frac{-\tau \mathcal{G}+\epsilon\|\vec{\varepsilon}\|_{\mathcal{H}}^2}{\mathcal{F}(D)}+\epsilon+2S\left| \frac{d}{dt}D\right|
\\
&\leq \left(-\tau+\frac{\epsilon}{C_1}+2\left|\frac{d}{dt}D\right|\right)S+\epsilon+\frac{\epsilon^2}{C_1}.
\end{align*}
Furthermore, since we have \eqref{sa}, by choosing $\epsilon$ sufficiently small, we may assume that we have for $t\geq T_{\epsilon}$
\begin{align*}
S^{\prime}\leq -\frac{\tau}{2}S+2\epsilon,
\end{align*}
which implies for $t>T_{\epsilon}$
\begin{align*}
S(t)\leq e^{-\frac{\tau(t-T_{\epsilon})}{2}}S(T_{\epsilon})+\frac{4\epsilon}{\tau}.
\end{align*}
Therefore we obtain
\begin{align}\label{Slimsup}
\limsup_{t\to\infty}\frac{ \mathcal{G}}{\mathcal{F}(D)}\leq 0.
\end{align}
By \eqref{C1es}, we also have for $t>T_{\epsilon}$
\begin{align*}
\mathcal{G}\geq -\epsilon\mathcal{F}(D),
\end{align*}
which implies
\begin{align}\label{Sliminf}
\liminf_{t\to\infty}\frac{ \mathcal{G}}{\mathcal{F}(D)}\geq 0.
\end{align}
By \eqref{Slimsup} and \eqref{Sliminf}, we obtain \eqref{Glim}. Thus we obtain \eqref{limep}. \eqref{Ves1} is easily derived from \eqref{Hamieq} and \eqref{limep}, and we complete the proof.

\end{proof}

\subsection{Estimates of the distance between each solitary waves}

In this subsection, we estimate $D$.

\begin{lemma}\label{Dupperbound}
Let $\vec{u}$ be a solution of \eqref{DNKG} satisfying \eqref{sa} for some $\sigma,\ z$, and $0<\delta<\delta_2$. Then there exists $C_{\star}>0$ and $C_{\ast}>0$ depending on $\vec{u}$ such that for all $t>e+1$, 
\begin{align}
\mathcal{F}(D(t))\geq \frac{C_{\star}}{t+1}, \label{mathcalFes}
\\
D(t)\leq \log{t}-\frac{d-1}{2}\log{(\log{t})}+C_{\ast}. \label{Des}
\end{align}
\end{lemma}
\begin{proof}
By \eqref{zes} and Lemma \ref{limeplem}, there exists $C_1>0$ such that 
\begin{align}\label{odeuekara}
\frac{d}{dt}D\leq C_1\mathcal{F}(D).
\end{align}
We define 
\begin{align*}
G(r)=\int_1^r \frac{ds}{\mathcal{F}(s)}.
\end{align*}
Then, we have for $r>1$
\begin{align}\label{Ghyouka}
\left|G(r)-\frac{1}{\mathcal{F}(r)}\right|\lesssim \frac{1}{r\mathcal{F}(r)}.
\end{align}
Therefore by \eqref{odeuekara} and \eqref{Ghyouka}, there exists $C_2>0$ such that
\begin{align*}
G(D(t))-G(D(0))\leq C_2t,
\end{align*}
which implies \eqref{mathcalFes}. Notably, if $R\gg1$ and $G(D)=R$, we have
\begin{align*}
D=\log{R}-\frac{d-1}{2}\log{(\log{R})}+C+O\left(\frac{\log{(\log{t})}}{\log{t}}\right).
\end{align*}
for some $C>0$. Therefore \eqref{Des} follows from \eqref{mathcalFes}, and we complete the proof.
\end{proof}

\subsection{Soliton repulsivity condition}

In this subsection, we assume the condition
\begin{align}\label{Vrep}
\liminf_{t\to \infty}\frac{V(t)}{\mathcal{F}(D(t))}> 0.
\end{align}
When $\vec{u}$ satisfies this condition, the estimate of $|a^+|$ is improved compared with Lemma \ref{a+lem1}.
\begin{lemma}\label{a+lem2}
Let $\vec{u}$ be a solution of \eqref{DNKG} satisfying \eqref{sa} for some $\sigma,\ z$, and $0<\delta<\delta_2$. Furthermore, we assume \eqref{Vrep}. Then there exists $C>0$ depending on $\vec{u}$ such that for $t\geq 0$,
\begin{align}\label{newa+es}
|a^+(t)|\leq C\left(\|\vec{\varepsilon}(t)\|_{\mathcal{H}}^2+\mathcal{F}(D(t))\right).
\end{align}
\end{lemma}
\begin{proof}
This proof is based on \cite[subsection 6.2]{IN}.

First, we estimate $a_k^-$ and $b_{k,l}$. By \eqref{apmes}, \eqref{bes}, and \eqref{limep}, we have
\begin{align}
|a_k^-(t)|&\lesssim e^{\nu^-t}|a_k^-(0)|+\int_0^te^{\nu^-(t-s)}\mathcal{F}(D(s))ds,\label{ak-es!}
\\
|b_{k,l}(t)|&\lesssim e^{-2\alpha t}|b_{k,l}(0)|+\int_0^t e^{-2\alpha(t-s)}\mathcal{F}(D(s))ds.\label{bkles!}
\end{align}
By \eqref{zes} and Lemma \ref{limeplem}, we have
\begin{align*}
\left|\frac{d}{dt}D\right|\lesssim \|\vec{\varepsilon}\|_{\mathcal{H}}^2+\mathcal{F}(D)\lesssim \mathcal{F}(D),
\end{align*}
and therefore we have
\begin{align}
e^{\nu^-t}+e^{-2\alpha t}\lesssim \mathcal{F}(D(t)).\label{Des!}
\end{align}
Furthermore, there exists $T>0$ such that 
\begin{align*}
\left| \frac{d}{dt}\mathcal{F}(D)\right|\leq \frac{\min{(-\nu^-,2\alpha)}}{2} \mathcal{F}(D).
\end{align*}
Then, we have for $t\geq s\geq T$,
\begin{align}
\mathcal{F}(D(s))\leq e^{\epsilon(t-s)}\mathcal{F}(D(t)),\label{Fes!}
\end{align}
where $\epsilon=\frac{\min{(-\nu^-,2\alpha)}}{2}$. Thus, by \eqref{ak-es!}, \eqref{bkles!}, \eqref{Des!}, and \eqref{Fes!}, we obtain
\begin{align}\label{a+nozokisim}
\sum_{k=1}^K \left(|a_k^-|^2+\left(\sum_{l=1}^d |b_{k,l}|^2\right)\right)\lesssim \mathcal{F}(D)^2.
\end{align}
Here, by \eqref{Vrep}, there exist $C_1>0$ and $T_1>0$ such that for $t>T_1$
\begin{align}\label{Vrepes1}
V(t)>C_1\mathcal{F}(D(t)).
\end{align}
Next, we collect some basic properties. By \eqref{apmes}, there exists $C_2>0$ such that for $1\leq k\leq K$ 
\begin{align*}
\left|\frac{d}{dt}a_k^+-\nu^+a_k^+\right|\leq C_2(\|\vec{\varepsilon}\|_{\mathcal{H}}^2+\mathcal{F}(D)).
\end{align*}
Therefore, there exists $C_3>0$ such that
\begin{align}\label{a+lemass}
\left| \frac{d}{dt}|a^+|^2-2\nu^+|a^+|^2 \right|\leq C_3|a^+|(\|\vec{\varepsilon}\|_{\mathcal{H}}^2+\mathcal{F}(D)).  
\end{align}
By \eqref{coersim}, \eqref{Lsumsim}, \eqref{Hamieq} \eqref{a+nozokisim}, and \eqref{Vrepes1}, there exist $0<C_4<1,\ C_5>0$ such that
\begin{align}\label{a+lemdouchi}
C_4(\|\vec{\varepsilon}\|_{\mathcal{H}}^2+\mathcal{F}(D))\leq E(\vec{u})-KE(Q,0)+C_5|a^+|^2\leq C_4^{-1}(\|\vec{\varepsilon}\|_{\mathcal{H}}^2+\mathcal{F}(D)).
\end{align}
Here, we define $M_1,\ M_2$ as
\begin{align}\label{Mdef}
M_1=\max{\left(C_5,\frac{C_3}{\nu^+} \right)}+1,\ M_2=\frac{2M_1}{C_4^2}.
\end{align}
Then, by \eqref{sa} and \eqref{lima}, there exist $\tilde{\delta}>0$ and  $T_2>T_1$ such that for $t>T_2$
\begin{align}
\|\vec{\varepsilon}\|_{\mathcal{H}}+\mathcal{F}(D)+|a^+|<\tilde{\delta}<\delta_2 \label{repasses1}
\\
|a^+|^2<\frac{C_4}{M_1+C_5}(\|\vec{\varepsilon}\|_{\mathcal{H}}^2+\mathcal{F}(D)). \label{repasses2}
\end{align}
Note that we may assume 
\begin{align}\label{deltabses}
\tilde{\delta}<\frac{1}{C_4M_1M_2}.
\end{align}
We assume that there exists $T_3>T_2$ such that 
\begin{align}\label{t3ass}
M_2(\|\vec{\varepsilon}(T_3)\|_{\mathcal{H}}^2+\mathcal{F}(D(T_3))<|a^+(T_3)|.
\end{align}
We then introduce the following bootstrap estimate
\begin{align}\label{Vrepbs}
M_1(\|\vec{\varepsilon}(t)\|_{\mathcal{H}}^2+\mathcal{F}(D(t))\leq |a^+(t)|<\tilde{\delta},
\end{align}
and we define $T_4>0$ as
\begin{align*}
T_4=\sup\{ t\in[T_3,\infty)\ \mbox{such that \eqref{Vrepbs} holds on}\ [T_3,t]\}.
\end{align*}
By \eqref{a+lemass}, \eqref{Mdef}, and  \eqref{Vrepbs}, we have for $T_3\leq t<T_4$
\begin{align*}
\frac{d}{dt}|a^+|^2\geq \nu^+|a^+|^2.
\end{align*}
In particular, $|a^+|$ is monotonically and exponentially growing. Here, we define 
\begin{align*}
\mathcal{J}=\frac{E(\vec{u})-KE(Q,0)+M_1|a^+|^2}{|a^+|}.
\end{align*}
Then by \eqref{a+lemdouchi}, \eqref{deltabses} and \eqref{t3ass}, we have
\begin{align}
\mathcal{J}(T_3)\leq \frac{C_4^{-1}(\|\vec{\varepsilon}(T_3)\|_{\mathcal{H}}^2+\mathcal{F}(D(T_3)))+M_1|a^+(T_3)|^2}{|a^+(T_3)|}<\frac{2}{M_2C_4}.
\end{align}
In addition, by \eqref{a+lemdouchi} and \eqref{repasses2}, we have
\begin{align*}
\frac{d}{dt}\mathcal{J}=\frac{-2\alpha\|{\partial}_tu\|_{L^2}^2}{|a^+|}+\frac{ M_1|a^+|^2-(E(\vec{u})-KE(Q,0))}{|a^+|^2}\frac{d}{dt}|a^+|\leq 0.
\end{align*}
Thus, we have
\begin{align*}
C_4(\|\vec{\varepsilon}\|_{\mathcal{H}}^2+\mathcal{F}(D))\leq \frac{2}{M_2C_4} |a^+|,
\end{align*}
and by \eqref{Mdef}, for $T_3\leq t<T_4$ we have
\begin{align}\label{bsleft}
M_1(\|\vec{\varepsilon}(t)\|_{\mathcal{H}}^2+\mathcal{F}(D(t))<|a^+(t)|.
\end{align}
Therefore as long as \eqref{Vrepbs} holds, $|a^+|$ increases exponentially, which implies that $T_4<\infty$ and $|a^+(T_4)|=\tilde{\delta}$. This contradicts \eqref{repasses1}. Thus we have for $t\geq T_1$
\begin{align*}
|a^+(t)|\leq M_2(\|\vec{\varepsilon}(t)\|_{\mathcal{H}}^2+\mathcal{F}(D(t)),
\end{align*}
and we complete the proof.
\end{proof}
From Lemma \ref{a+lem2}, we improve the estimate of $\|\vec{\varepsilon}\|_{\mathcal{H}}$. Furthermore,  we improve the estimate of the centers of the solitary wave.

\begin{lemma}\label{epzes}
Let $\vec{u}$ be a solution of \eqref{DNKG} satisfying \eqref{sa} for some $\sigma, z$, and $0<\delta<\delta_2$. Furthermore, we assume \eqref{Vrep}. Then, there exists $C>0$ depending on $\vec{u}$ such that 
\begin{align}\label{newepes}
\|\vec{\varepsilon}(t)\|_{\mathcal{H}}\leq C \mathcal{F}(D(t)).
\end{align}
Furthermore, for any $1<\theta<\theta_{\star}$, there exists $\tilde{C}>0$ such that for $1\leq k\leq K$, we have
\begin{align}\label{newzes}
\left|\dot{z}_k+\sum_{1\leq i\leq K,\ i\neq k}\sigma_i\sigma_k\mathcal{F}(|z_k-z_i|)\frac{z_k-z_i}{|z_k-z_i|}\right|&\leq \tilde{C} e^{-\theta D}.
\end{align}
\end{lemma}
\begin{proof}
By \eqref{Gdouchi} and \eqref{newa+es} there exists $L_1>0$ such that 
\begin{align*}
\mathcal{G}+L_1\mathcal{F}(D)^2\sim \|\vec{\varepsilon}\|_{\mathcal{H}}^2+\mathcal{F}(D)^2.
\end{align*}
Furthermore, by \eqref{newa+es} and Lemma \ref{Glem}, there exists $\tau_1>0$ such that 
\begin{align}\label{newGes1}
\frac{d}{dt}\mathcal{G}\leq -\tau_1 \mathcal{G}+\frac{1}{\tau_1} \mathcal{F}(D)^2.
\end{align}
By \eqref{zes}, we have
\begin{align}\label{newzesass1}
\frac{d}{dt}\mathcal{F}(D)^2\leq \frac{\tau_1}{2}\mathcal{F}(D)^2.
\end{align}
By \eqref{newGes1} and \eqref{newzesass1}, we obtain
\begin{align*}
\mathcal{G}\lesssim e^{-\tau_1t}\mathcal{G}(0)+\mathcal{F}(D)^2,
\end{align*}
which implies \eqref{newepes}. \eqref{newzes} follows from \eqref{zes} and \eqref{newepes}. Therefore we complete the proof.
\end{proof}

\section{Global dynamics of the center of each solitary wave}

Here, we consider the $K=3$ case. In \cite{CD}, it has been proven that a 3-solitary wave that has the same sign does not exist. Therefore, when we consider a 3-solitary wave, we may assume
\begin{align}\label{3soliass}
K=3, \sigma_1=\sigma_2=-\sigma_3.
\end{align}
Furthermore, we define
\begin{align*}
Z_1=z_1-z_3,\ Z_2=z_2-z_3,\ Z_0=z_2-z_1,
\\
D_1=|Z_1|,\ D_2=|Z_2|, D_0=|Z_0|.
\end{align*}
By the definition, we have
\begin{align}\label{Ddef2}
D=\min{(D_0,D_1,D_2)}.
\end{align}
In this section, we investigate the long-time dynamics of a 3-solitary wave. Although there are several possible ways to define the centers of the solitons, the long-time behavior is determined once the leading interaction term has been identified. Thus, we define $\theta_1$ to be a constant greater than $1$ and sufficiently close to it. In particular, $\theta_1$ satisfies 
\begin{align}\label{theta1}
1<\theta_1<\theta_{\star}.
\end{align}
\begin{remark}
We note that the closeness of $\theta_1$ to $1$ depends on the solution $\vec{u}$. This is because, in a later subsection, the smallness of $\theta_1-1$ will be required when applying Lemma \ref{Dupperbound}. However, since we fix $\vec{u}$ and consider its behavior as $t\to\infty$, this dependence on $\vec{u}$ is not an issue.
\end{remark}

Then, Lemma \ref{expansionlem} can be rewritten as follows.

\begin{lemma}\label{3solicenter1}
Let $\vec{u}$ be a solution of \eqref{DNKG} satisfying \eqref{sa} and \eqref{3soliass} for some $z,\ \sigma$, and $0<\delta\ll 1$.\ Then we have 
\begin{align}
\left|\dot{Z}_1-2\mathcal{F}(D_1)\frac{Z_1}{D_1}-\mathcal{F}(D_2)\frac{Z_2}{D_2}-\mathcal{F}(D_0)\frac{Z_0}{D_0}\right|\lesssim e^{-\theta_1 D}+\|\vec{\varepsilon}\|_{\mathcal{H}}^2,\label{Z1es}
\\
\left|\dot{Z}_2-2\mathcal{F}(D_2)\frac{Z_2}{D_2}-\mathcal{F}(D_1)\frac{Z_1}{D_1}+\mathcal{F}(D_0)\frac{Z_0}{D_0}\right|\lesssim e^{-\theta_1 D}+\|\vec{\varepsilon}\|_{\mathcal{H}}^2,\label{Z2es}
\\
\left|\dot{Z}_0+2\mathcal{F}(D_0)\frac{Z_0}{D_0}+\mathcal{F}(D_1)\frac{Z_1}{D_1}-\mathcal{F}(D_2)\frac{Z_2}{D_2}\right|\lesssim e^{-\theta_1 D}+\|\vec{\varepsilon}\|_{\mathcal{H}}^2.\label{Z0es}
\end{align}
Furthermore, we have
\begin{align}
\left|\frac{d}{dt}D_1-2\mathcal{F}(D_1)-\mathcal{F}(D_2)\frac{D_1}{D_2}-\left(\frac{\mathcal{F}(D_2)}{D_1D_2}+\frac{\mathcal{F}(D_0)}{D_0D_1}\right)(Z_0\cdot Z_1)\right|\lesssim e^{-\theta_1 D}+\|\vec{\varepsilon}\|_{\mathcal{H}}^2,\label{D1es}
\\
\left|\frac{d}{dt}D_2-2\mathcal{F}(D_2)-\mathcal{F}(D_1)\frac{D_2}{D_1}+\left(\frac{\mathcal{F}(D_1)}{D_1D_2}+\frac{\mathcal{F}(D_0)}{D_0D_2}\right)(Z_0\cdot Z_2)\right|\lesssim e^{-\theta_1 D}+\|\vec{\varepsilon}\|_{\mathcal{H}}^2, \label{D2es}
\\
\left|\frac{d}{dt}D_0+2\mathcal{F}(D_0)-\mathcal{F}(D_2)\frac{D_0}{D_2}-\left(\frac{\mathcal{F}(D_2)}{D_0D_2}-\frac{\mathcal{F}(D_1)}{D_0D_1}\right)(Z_0\cdot Z_1)\right|\lesssim e^{-\theta_1 D}+\|\vec{\varepsilon}\|_{\mathcal{H}}^2, \label{D0es1}
\\
\left|\frac{d}{dt}D_0+2\mathcal{F}(D_0)-\mathcal{F}(D_1)\frac{D_0}{D_1}-\left(\frac{\mathcal{F}(D_2)}{D_0D_2}-\frac{\mathcal{F}(D_1)}{D_0D_1}\right)(Z_0\cdot Z_2)\right|\lesssim e^{-\theta_1 D}+\|\vec{\varepsilon}\|_{\mathcal{H}}^2. \label{D0es2}
\end{align}
\end{lemma}

\begin{proof}
First, by Lemma \ref{expansionlem} and \eqref{3soliass}, we have
\begin{align*}
\dot{z}_1&=\mathcal{F}(D_1)\frac{Z_1}{D_1}+\mathcal{F}(D_0)\frac{Z_0}{D_0}+O(e^{-\theta_1 D}+\|\vec{\varepsilon}\|_{\mathcal{H}}^2),
\\
\dot{z}_2&=\mathcal{F}(D_2)\frac{Z_2}{D_2}-\mathcal{F}(D_0)\frac{Z_0}{D_0}+O(e^{-\theta_1 D}+\|\vec{\varepsilon}\|_{\mathcal{H}}^2),
\\
\dot{z}_3&=-\mathcal{F}(D_1)\frac{Z_1}{D_1}-\mathcal{F}(D_2)\frac{Z_2}{D_2}+O(e^{-\theta_1 D}+\|\vec{\varepsilon}\|_{\mathcal{H}}^2).
\end{align*}
Thus, by the above estimates, we obtain \eqref{Z1es}, \eqref{Z2es}, \eqref{Z0es}. Furthermore, by the direct computation, we have
\begin{align*}
\frac{d}{dt}D_1&=\frac{\dot{Z}_1\cdot Z_1}{D_1}
\\
&=2\mathcal{F}(D_1)+\mathcal{F}(D_2)\frac{Z_1\cdot Z_2}{D_1D_2}+\mathcal{F}(D_0)\frac{Z_0\cdot Z_1}{D_0D_1}+O(e^{-\theta_1 D}+\|\vec{\varepsilon}\|_{\mathcal{H}}^2)
\\
&=2\mathcal{F}(D_1)+\frac{\mathcal{F}(D_2)}{D_1D_2}(Z_0+Z_1)\cdot Z_1+\frac{\mathcal{F}(D_0)}{D_0D_1}(Z_0\cdot Z_1)+O(e^{-\theta_1 D}+\|\vec{\varepsilon}\|_{\mathcal{H}}^2)
\\
&=2\mathcal{F}(D_1)+\frac{\mathcal{F}(D_2)D_1}{D_2}+\left(\frac{\mathcal{F}(D_2)}{D_1D_2}+\frac{\mathcal{F}(D_0)}{D_0D_1}\right)(Z_0\cdot Z_1)+O(e^{-\theta_1 D}+\|\vec{\varepsilon}\|_{\mathcal{H}}^2),
\\
\frac{d}{dt}D_2&=\frac{\dot{Z}_2\cdot Z_2}{D_2}
\\
&=2\mathcal{F}(D_2)+\mathcal{F}(D_1)\frac{Z_1\cdot Z_2}{D_1D_2}-\mathcal{F}(D_0)\frac{Z_0\cdot Z_2}{D_0D_2}+O(e^{-\theta_1 D}+\|\vec{\varepsilon}\|_{\mathcal{H}}^2)
\\
&=2\mathcal{F}(D_2)+\frac{\mathcal{F}(D_1)}{D_1D_2}(Z_2-Z_0)\cdot Z_2-\frac{\mathcal{F}(D_0)}{D_0D_2}(Z_0\cdot Z_2)+O(e^{-\theta_1 D}+\|\vec{\varepsilon}\|_{\mathcal{H}}^2)
\\
&=2\mathcal{F}(D_2)+\frac{\mathcal{F}(D_1)D_2}{D_1}-\left(\frac{\mathcal{F}(D_1)}{D_1D_2}+\frac{\mathcal{F}(D_0)}{D_0D_2}\right)(Z_0\cdot Z_2)+O(e^{-\theta_1 D}+\|\vec{\varepsilon}\|_{\mathcal{H}}^2).
\end{align*}
Furthermore, $D_0$ satisfies
\begin{align*}
\frac{d}{dt}D_0&=\frac{\dot{Z}_0\cdot Z_0}{D_0}
\\
&=-2\mathcal{F}(D_0)-\mathcal{F}(D_1)\frac{Z_0\cdot Z_1}{D_0D_1}+\mathcal{F}(D_2)\frac{Z_0\cdot Z_2}{D_0D_2}+O(e^{-\theta_1 D}+\|\vec{\varepsilon}\|_{\mathcal{H}}^2).
\end{align*}
Thus, we obtain
\begin{align*}
\frac{d}{dt}D_0&=-2\mathcal{F}(D_0)-\mathcal{F}(D_1)\frac{Z_0\cdot Z_1}{D_0D_1}+\mathcal{F}(D_2)\frac{Z_0\cdot Z_2}{D_0D_2}+O(e^{-\theta_1 D}+\|\vec{\varepsilon}\|_{\mathcal{H}}^2)
\\
&=-2\mathcal{F}(D_0)-\frac{\mathcal{F}(D_1)}{D_0D_1}(Z_0\cdot Z_1)+\frac{\mathcal{F}(D_2)}{D_0D_2}(Z_0+Z_1)\cdot Z_0+O(e^{-\theta_1 D}+\|\vec{\varepsilon}\|_{\mathcal{H}}^2)
\\
&=-2\mathcal{F}(D_0)+\frac{\mathcal{F}(D_2)D_0}{D_2}+\left(\frac{\mathcal{F}(D_2)}{D_0D_2}-\frac{\mathcal{F}(D_1)}{D_0D_1}\right)(Z_0\cdot Z_1)+O(e^{-\theta_1 D}+\|\vec{\varepsilon}\|_{\mathcal{H}}^2),
\end{align*}
and
\begin{align*}
\frac{d}{dt}D_0&=-2\mathcal{F}(D_0)-\mathcal{F}(D_1)\frac{Z_0\cdot Z_1}{D_0D_1}+\mathcal{F}(D_2)\frac{Z_0\cdot Z_2}{D_0D_2}+O(e^{-\theta_1 D}+\|\vec{\varepsilon}\|_{\mathcal{H}}^2)
\\
&=-2\mathcal{F}(D_0)-\frac{\mathcal{F}(D_1)}{D_0D_1}(Z_2-Z_0)\cdot Z_0+\frac{\mathcal{F}(D_2)}{D_0D_2}(Z_0\cdot Z_2)+O(e^{-\theta_1 D}+\|\vec{\varepsilon}\|_{\mathcal{H}}^2)
\\
&=-2\mathcal{F}(D_0)+\frac{\mathcal{F}(D_1)D_0}{D_1}+\left(\frac{\mathcal{F}(D_2)}{D_0D_2}-\frac{\mathcal{F}(D_1)}{D_0D_1}\right)(Z_0\cdot Z_2)+O(e^{-\theta_1 D}+\|\vec{\varepsilon}\|_{\mathcal{H}}^2).
\end{align*}
By summarizing these discussions, this completes the proof.
\end{proof}

Moreover, by Lemma \ref{epzes}, the following lemma also holds.

\begin{lemma}\label{repcondes}
Let $\vec{u}$ be a solution of \eqref{DNKG} satisfying \eqref{sa} and \eqref{3soliass} for some $z,\ \sigma$, and $0<\delta<\delta_2$. Furthermore, we assume \eqref{Vrep}. Then, we have
\begin{align}
\left|\dot{Z}_1-2\mathcal{F}(D_1)\frac{Z_1}{D_1}-\mathcal{F}(D_2)\frac{Z_2}{D_2}-\mathcal{F}(D_0)\frac{Z_0}{D_0}\right|\lesssim e^{-\theta_1 D},\label{Z1es2}
\\
\left|\dot{Z}_2-2\mathcal{F}(D_2)\frac{Z_2}{D_2}-\mathcal{F}(D_1)\frac{Z_1}{D_1}+\mathcal{F}(D_0)\frac{Z_0}{D_0}\right|\lesssim e^{-\theta_1 D},\label{Z2es2}
\\
\left|\dot{Z}_0+2\mathcal{F}(D_0)\frac{Z_0}{D_0}+\mathcal{F}(D_1)\frac{Z_1}{D_1}-\mathcal{F}(D_2)\frac{Z_2}{D_2}\right|\lesssim e^{-\theta_1 D}.\label{Z0es2}
\end{align}
Furthermore, we have
\begin{align}
\left|\frac{d}{dt}D_1-2\mathcal{F}(D_1)-\mathcal{F}(D_2)\frac{D_1}{D_2}-\left(\frac{\mathcal{F}(D_2)}{D_1D_2}+\frac{\mathcal{F}(D_0)}{D_0D_1}\right)(Z_0\cdot Z_1)\right|\lesssim e^{-\theta_1 D},\label{D1es2}
\\
\left|\frac{d}{dt}D_2-2\mathcal{F}(D_2)-\mathcal{F}(D_1)\frac{D_2}{D_1}+\left(\frac{\mathcal{F}(D_1)}{D_1D_2}+\frac{\mathcal{F}(D_0)}{D_0D_2}\right)(Z_0\cdot Z_2)\right|\lesssim e^{-\theta_1 D}, \label{D2es2}
\\
\left|\frac{d}{dt}D_0+2\mathcal{F}(D_0)-\mathcal{F}(D_2)\frac{D_0}{D_2}-\left(\frac{\mathcal{F}(D_2)}{D_0D_2}-\frac{\mathcal{F}(D_1)}{D_0D_1}\right)(Z_0\cdot Z_1)\right|\lesssim e^{-\theta_1 D}, \label{D0es2-1}
\\
\left|\frac{d}{dt}D_0+2\mathcal{F}(D_0)-\mathcal{F}(D_1)\frac{D_0}{D_1}-\left(\frac{\mathcal{F}(D_2)}{D_0D_2}-\frac{\mathcal{F}(D_1)}{D_0D_1}\right)(Z_0\cdot Z_2)\right|\lesssim e^{-\theta_1 D}. \label{D0es2-2}
\end{align}
\end{lemma}
\begin{proof}
By Lemma \ref{epzes}, we obtain \eqref{Z1es2}-\eqref{D0es2-2} using the same argument as the proof of Lemma \ref{3solicenter1}.

\end{proof}

\subsection{Considerations on the location of the center}

Here, we estimate $D_0,\ D_1,\ D_2$. Before estimating $D_1$ and $D_2$, we define
\begin{align}
\tilde{D}=\min{(D_1,D_2)}.
\end{align}
Then the following lemma holds.

\begin{lemma}\label{dtildes}
There exist $\epsilon_1>0$ and $\delta_3>0$ depending only on $d,\ p,\ \alpha$ such that the following holds. Let $\vec{u}$ be a solution of \eqref{DNKG} satisfying \eqref{sa} and \eqref{3soliass} for some $z,\ \sigma$, and $0<\delta<\delta_3$. Furthermore, we assume that there exist $0\leq T_1\leq T_2\leq \infty$ such that for $T_1<t<T_2$ 
\begin{align}
\|\vec{\varepsilon}(t)\|_{\mathcal{H}}^2<\epsilon_1 \mathcal{F}(D(t)), \label{varepes1}
\\
|\tilde{D}(t)-D_0(t)|<\sqrt{D}(t).\label{Dass}
\end{align}
Then, for $T_1<t<T_2$, we have
\begin{align}\label{Dest}
\left( \tilde{D}(t)-D_0(t)\right)-\left( \tilde{D}(T_1)-D_0(T_1)\right)\gtrsim \int_{T_1}^t \frac{ds}{(s-T_1)+1}.
\end{align}
\end{lemma}
\begin{proof}
First, we assume that $D_1(t)<D_2(t)$. Then by \eqref{D1es} and \eqref{D0es1}, we have
\begin{align}\label{sanoes1}
\begin{aligned}
\frac{d}{dt}(D_1-D_0)&=2(\mathcal{F}(D_1)+\mathcal{F}(D_0))+\frac{\mathcal{F}(D_2)}{D_2}(D_1-D_0)
\\
&\quad +\frac{\mathcal{F}(D_2)(D_0-D_1)}{D_0D_1D_2}(Z_0\cdot Z_1)+\frac{\mathcal{F}(D_0)+\mathcal{F}(D_1)}{D_0D_1}(Z_0\cdot Z_1)
\\
&\quad +O(e^{-\theta_1 D}+\|\vec{\varepsilon}\|_{\mathcal{H}}^2).
\end{aligned}
\end{align}
By using the Cauchy-Schwarz inequality, we have
\begin{align*}
-D_0D_1\leq Z_0\cdot Z_1\leq D_0D_1.
\end{align*}
By \eqref{Dass}, there exists $C>0$ depending only on $d,\ p,\ \alpha$ such that
\begin{align*}
\frac{d}{dt}(D_1-D_0)&\geq 2(\mathcal{F}(D_1)+\mathcal{F}(D_0))-(\mathcal{F}(D_0)+\mathcal{F}(D_1))
\\
&\quad-\frac{2\mathcal{F}(D_2)}{\sqrt{D_2}}-C\left(e^{-\theta_1 D}+\|\vec{\varepsilon}\|_{\mathcal{H}}^2\right).
\end{align*}
Notably, we have $D=\min{(D_0,D_1)}$. In addition, when we choose $\epsilon_1>0$ and $\delta_3>0$ sufficiently small, by \eqref{sa}, and \eqref{varepes1}, we obtain
\begin{align}\label{Dlemes1}
\frac{2\mathcal{F}(D_2)}{\sqrt{D_2}}+C\left(e^{-\theta_1 D}+\|\vec{\varepsilon}\|_{\mathcal{H}}^2\right)<\frac{1}{2}\mathcal{F}(D).
\end{align}
We note that by this argument, $\epsilon_1$ and $\delta_3$ depend only on $d,\ p,\ \alpha$. By \eqref{Dlemes1}, we obtain 
\begin{align}\label{Dnosaesa}
\frac{d}{dt}(D_1-D_0)&\geq \frac{1}{2}\mathcal{F}(D).
\end{align}
Second, we assume $D_2\leq D_1$. Then by a similar calculation, we obtain 
\begin{align}\label{Dnosaesb}
\frac{d}{dt}(D_2-D_0)&\geq \frac{1}{2}\mathcal{F}(D).
\end{align}
By Lemma \ref{Dupperbound}, \eqref{Dnosaesa}, and \eqref{Dnosaesb}, we have
\begin{align}\label{Dnosaesc}
\frac{d}{dt}(\tilde{D}-D_0)\gtrsim \mathcal{F}(D)\gtrsim \frac{1}{t-T_1+1}.
\end{align}
By \eqref{Dnosaesc}, we complete the proof.

\end{proof}

\subsection{The distance between solitons with the same sign}
Here, we consider a 3-solitary wave. In this subsection, we show that the distance between solitons with the same sign is sufficiently large.
\begin{lemma}\label{repcondlem}
Let $\vec{u}$ be a solution of \eqref{DNKG} satisfying \eqref{sa} and \eqref{3soliass} for some $z,\ \sigma$, and $0<\delta<\delta_3$. Then, we have
\begin{align}\label{d0long}
\lim_{t\to\infty}\left(D_0(t)-\tilde{D}(t)\right)=\infty.
\end{align}
\end{lemma}
\begin{proof}
We fix $M>0$ and assume that there exists a sequence $t_n\to \infty$ such that 
\begin{align}\label{replemass1}
D_0(t_n)\leq \tilde{D}(t_n)+M.
\end{align}
Then, we choose $N\in \mathbb{N}$ that $t_N$ satisfies for all $t>t_N$
\begin{align}\label{epsmall}
\|\vec{\varepsilon}\|_{\mathcal{H}}^2\ll \epsilon_1\mathcal{F}(D)
\end{align}
and \eqref{Dass} hold. Then we introduce  the following bootstrap estimate
\begin{align}\label{belem1}
-M\leq \tilde{D}-D_0\leq c_2,
\end{align}
where $c_2>0$ is a constant such that for all $\tilde{D}-D_0\geq c_2$
\begin{align}\label{repasslem2}
\mathcal{F}(\tilde{D})<\frac{1}{4}\mathcal{F}(D_0).
\end{align}
We define $T\in [t_N,\infty]$ by
\begin{align}\label{bsdef1}
T=\sup{\{ t\in[t_N, \infty)\ \mbox{such that holds}\ \eqref{belem1}\ \mbox{on}\ [t_N,t] \}} .
\end{align}
Furthermore, we assume that $T<\infty$. Then we have
\begin{align*}
\tilde{D}(T)-D_0(T)=-M\ \mbox{or}\ \tilde{D}(T)-D_0(T)=c_2.
\end{align*}
By Lemma \ref{dtildes}, if $\tilde{D}(T)-D_0(T)=-M$, we have $\tilde{D}(t)-D_0(t)\geq -M$ for $T<t<T+\epsilon$, where $0<\epsilon\ll 1$ and it contradicts to \eqref{bsdef1}. If $\tilde{D}(T)-D_0(T)=c_2$, by the definition of $c_2$, we have
\begin{align*}
\mathcal{F}(D_1(T))+\mathcal{F}(D_2(T))-\mathcal{F}(D_0(T))\leq 2\mathcal{F}(\tilde{D}(T))-\mathcal{F}(D(T))<-\frac{1}{2}\mathcal{F}(D(T)).
\end{align*}
Then by Lemma \ref{Hamilem} and \eqref{epsmall}, we have $E(\vec{u}(T))<3E(Q,0)$, which contradicts that $\vec{u}$ is a 3-solitary wave. Therefore, we obtain $T=\infty$. On the other hand, by Lemma \ref{dtildes}, we have for all $t_N<t$,
\begin{align*}
\left(\tilde{D}(t)-D_0(t)\right)-\left(\tilde{D}(t_N)-D_0(t_N)\right)\gtrsim \int_{t_N}^t \frac{ds}{s-t_N+1}.
\end{align*}
When we consider $t\to\infty$, we obtain $\tilde{D}(t)-D_0(t)\to \infty$, which contradicts to \eqref{belem1}. Gathering these arguments, we complete the proof.
\end{proof}

In particular, when $D_0\geq \tilde{D}+M$ holds for a large constant $M>0$, we have
\begin{align*}
\frac{1}{2}\mathcal{F}(\tilde{D})=\frac{1}{2}\mathcal{F}(D)>\mathcal{F}(D_0).
\end{align*}
Therefore we obtain for $t\gg 1$
\begin{align*}
\mathcal{F}(D_1)+\mathcal{F}(D_2)-\mathcal{F}(D_0)> \frac{1}{2}\mathcal{F}(D).
\end{align*}
Therefore we obtain the following lemma:
\begin{lemma}\label{3solirep}
Let $\vec{u}$ be a solution of \eqref{DNKG} satisfying \eqref{sa} and \eqref{3soliass} for some $z,\ \sigma$, and $0<\delta<\delta_3$. Then $\vec{u}$ satisfies \eqref{Vrep}.
\end{lemma}
\begin{proof}
By Lemma \ref{repcondlem}, we have for large $t$,
\begin{align*}
\mathcal{F}(D_0)<\frac{1}{2}\mathcal{F}(D).
\end{align*}
Thus, we have
\begin{align*}
V=\mathcal{F}(D_1)+\mathcal{F}(D_2)-\mathcal{F}(D_0)>\frac{1}{2}\mathcal{F}(D),
\end{align*}
which implies \eqref{Vrep}.
\end{proof}

\subsection{Dynamics of 3-solitons under the repulsive condition}
In this subsection, we estimate the difference between $D_1$ and $D_2$. We define $\theta_2$ as
\begin{align}\label{theta2}
\theta_2=\frac{\theta_1-1}{2}>0.
\end{align}
\begin{lemma}\label{Dnosalem}
Let $\vec{u}$ be a solution of \eqref{DNKG} satisfying \eqref{sa} and \eqref{3soliass} for some $z,\ \sigma$, and $0<\delta<\delta_3$. Then, we have
\begin{align}\label{Dnosaes}
\left|D_1(t)-D_2(t)\right|\lesssim t^{-\theta_2}.
\end{align}
Furthermore, we have 
\begin{align}
|D_1(t)-D(t)|+|D_2(t)-D(t)|&\lesssim t^{-\theta_2},\label{Dnosaes2}
\\
\mathcal{F}(D(t))&\sim \frac{1}{t+1}.\label{logtorder}
\end{align}
\end{lemma}
\begin{proof}
By Lemma \ref{3solirep}, we have \eqref{Vrep}. Then, by Lemma \ref{repcondes} and $Z_2-Z_1=Z_0$, we have
\begin{align*}
\frac{d}{dt}D_1&=2\mathcal{F}(D_1)+\mathcal{F}(D_2)\frac{D_1}{D_2}+\left( \frac{\mathcal{F}(D_2)}{D_1D_2}+\frac{\mathcal{F}(D_0)}{D_0D_1}\right)(Z_0\cdot Z_1)+O(e^{-\theta_1D})
\\
&=2\mathcal{F}(D_1)-\mathcal{F}(D_0)\frac{D_1}{D_0}+\left(\frac{\mathcal{F}(D_2)}{D_1D_2}+\frac{\mathcal{F}(D_0)}{D_0D_1}\right)(Z_1\cdot Z_2)+O(e^{-\theta_1 D}),
\\
\frac{d}{dt}D_2&=2\mathcal{F}(D_2)+\mathcal{F}(D_1)\frac{D_2}{D_1}-\left(\frac{\mathcal{F}(D_1)}{D_1D_2}+\frac{\mathcal{F}(D_0)}{D_0D_2}\right)(Z_0\cdot Z_2)+O(e^{-\theta_1 D})
\\
&=2\mathcal{F}(D_2)-\mathcal{F}(D_0)\frac{D_2}{D_0}+\left(\frac{\mathcal{F}(D_1)}{D_1D_2}+\frac{\mathcal{F}(D_0)}{D_0D_2}\right)(Z_1\cdot Z_2)+O(e^{-\theta_1 D}).
\end{align*}
Therefore we obtain
\begin{align}\label{d1d2saes}
\begin{aligned}
\frac{d}{dt}(D_1-D_2)&=\left( \mathcal{F}(D_1)-\mathcal{F}(D_2)\right)\left(2-\frac{Z_1\cdot Z_2}{D_1D_2}\right)
\\
&\quad -\frac{\mathcal{F}(D_0)}{D_0}\left(1+\frac{Z_1\cdot Z_2}{D_1D_2}\right)(D_1-D_2)+O(e^{-\theta_1 D}).
\end{aligned}
\end{align}
The proof of Lemma \ref{Dnosalem} proceeds in three steps.

First, we show that there exist $T>0$ and a small constant $\epsilon>0$ such that for all $t>T$
\begin{align}\label{upperepsilon}
|D_1(t)-D_2(t)|\leq \epsilon.
\end{align}
By Lemma \ref{repcondlem}, $\tilde{D}=D$ holds true for sufficiently large $t>0$. Furthermore there exists $T_1>0$ such that for $t>T_1$
\begin{align}\label{D0okii}
\mathcal{F}(D_0)<\frac{1}{4}\mathcal{F}(D).
\end{align}
Now, we assume that there exists a sequence $t_n\to\infty$ such that 
\begin{align*}
|D_1(t_n)-D_2(t_n)|>\epsilon.
\end{align*}
Then we choose large $N\in \mathbb{N}$ satisfying $e^{-\theta_1 D}<\epsilon^2\mathcal{F}(D)$ for $t>t_N>T_1$, and introduce the following bootstrap estimate
\begin{align}\label{d1d2sabs}
|D_1(t)-D_2(t)|\geq \epsilon.
\end{align}
We define $T_2\in [t_N,\infty]$ as 
\begin{align}\label{bsdef2}
T_2=\sup{\{ t\in[t_N, \infty)\ \mbox{such that holds}\ \eqref{d1d2sabs}\ \mbox{on}\ [t_N,t] \}} .
\end{align}
We assume $T_2<\infty$. Then $D_1$ and $D_2$ satisfy
\begin{align*}
|D_1(T_2)-D_2(T_2)|=\epsilon.
\end{align*}
Then by \eqref{d1d2saes} we have $\left(\frac{d}{dt}|D_1-D_2|\right)(T_2)\leq  0$, which contradicts to \eqref{bsdef2}. Thus we obtain $T_2=\infty$ and we obtain \eqref{upperepsilon} for $t\gg 1$.

Second, we show \eqref{logtorder}. By Lemma \ref{repcondlem}, and \eqref{upperepsilon}, if $D_1\leq D_2$ for $t\gg 1$ we have $D=D_1$ and $D_2<D_0$ and 
\begin{align*}
\frac{d}{dt}D_1\geq 2\mathcal{F}(D_1)-2\mathcal{F}(D_0)-\mathcal{F}(D_2)-Ce^{-\theta_1 D}
\end{align*}
for some $C>0$. Furthermore by \eqref{D0okii} and $e^{-\theta_1 D}\ll \mathcal{F}(D)$ for $t\gg1$, we have for $t\gg 1$
\begin{align*}
\frac{d}{dt}D_1\geq \frac{1}{3}\mathcal{F}(D_1).
\end{align*}
By the same argument, if $D_2\leq D_1$ we have for $t\gg 1$
\begin{align*}
\frac{d}{dt}D_2\geq \frac{1}{3}\mathcal{F}(D_2).
\end{align*}
Thus, we have for $t\gg 1$
\begin{align*}
\frac{d}{dt}D\geq \frac{1}{3}\mathcal{F}(D).
\end{align*}
Therefore we have 
\begin{align*}
\frac{d}{dt}D\sim \mathcal{F}(D),
\end{align*}
which implies \eqref{logtorder}.

Last, we show \eqref{Dnosaes} and \eqref{Dnosaes2}. By \eqref{Fesb}, \eqref{mathcalFes}, and \eqref{upperepsilon}, we have
\begin{align}\label{mathcalFsimrev}
\mathcal{F}(D_1(t))-\mathcal{F}(D_2(t))\sim \frac{1}{t+1}\left(D_2(t)-D_1(t)\right).
\end{align}
By \eqref{d1d2saes} and \eqref{mathcalFsimrev}, there exist $c_1>0$ and $c_2>0$ such that for $t\gg 1$ 
\begin{align}\label{Dnosaesr1}
\frac{d}{dt}|D_1-D_2|\leq -\frac{c_1}{t}|D_1-D_2|+c_2t^{-\theta_2-1}.
\end{align}
We note that $\theta_2$ satisfies
\begin{align*}
e^{-\theta_1 D}\lesssim \mathcal{F}(D)^{\theta_2+1},
\end{align*}
and \eqref{Dnosaesr1} holds true. We also note that $c_1$ is independent of $\theta_2$. Then, by the Gronwall's inequality, we have for $t>T\gg 1$
\begin{align*}
|D_1(t)-D_2(t)|\lesssim t^{-c_1}+t^{-c_1}\int_T^t s^{c_1-\theta_2-1}ds\lesssim t^{-\min{(c_1,\theta_2)}}=t^{-\theta_2},
\end{align*}
since $\theta_2$ is close to $0$. Thus, we obtain \eqref{Dnosaes}. The estimate \eqref{Dnosaes2} follows directly from \eqref{Dnosaes} and Lemma \ref{repcondlem}, and we complete the proof.
\end{proof}
Here, we again compute $Z_0$, $Z_1$, and $Z_2$. We define $\theta_3$ as
\begin{align}\label{theta3}
\theta_3=\theta_2+1=\frac{\theta_1+1}{2}>1.
\end{align}
\begin{lemma}\label{3soliZesrefine}
Let $\vec{u}$ be a solution of \eqref{DNKG} satisfying \eqref{sa} and \eqref{3soliass} for some $z,\ \sigma$, and $0<\delta<\delta_3$. Then, we have for $t>0$
\begin{align}
\left|\dot{Z}_1-\frac{\mathcal{F}(D)}{D}(2Z_1+Z_2)-\frac{\mathcal{F}(D_0)}{D_0}Z_0\right|&\lesssim t^{-\theta_3},\label{Z1refes}
\\
\left|\dot{Z}_2-\frac{\mathcal{F}(D)}{D}(Z_1+2Z_2)+\frac{\mathcal{F}(D_0)}{D_0}Z_0\right|&\lesssim t^{-\theta_3}, \label{Z2refes}
\\
\left|\dot{Z}_0-\left(\frac{\mathcal{F}(D)}{D}-\frac{2\mathcal{F}(D_0)}{D_0}\right)Z_0\right|&\lesssim t^{-\theta_3}.\label{Z0refes}
\end{align}
\end{lemma}
\begin{proof}
By Lemma \ref{Dnosalem}, we have the following estimates:
\begin{align}\label{Dninaosu}
\begin{aligned}
\left| \mathcal{F}(D_1)-\mathcal{F}(D)\right|+\left|\mathcal{F}(D_2)-\mathcal{F}(D)\right|&\lesssim t^{-\theta_3},
\\
\left|\frac{\mathcal{F}(D_1)}{D_1}-\frac{\mathcal{F}(D)}{D}\right|+\left|\frac{\mathcal{F}(D_2)}{D_2}-\frac{\mathcal{F}(D)}{D}\right|&\lesssim \frac{t^{-\theta_3}}{D}.
\end{aligned}
\end{align}
By \eqref{Dninaosu}, we have
\begin{align}\label{D1D2Ddekinji}
\begin{aligned}
\mathcal{F}(D_1)\frac{Z_1}{D_1}&=\mathcal{F}(D)\frac{Z_1}{D}+O(t^{-\theta_3}),
\\
\mathcal{F}(D_2)\frac{Z_2}{D_2}&=\mathcal{F}(D)\frac{Z_2}{D}+O(t^{-\theta_3}).
\end{aligned}
\end{align}
By Lemma \ref{3solirep}, we have \eqref{Vrep}. Thus by Lemma \ref{repcondes} and \eqref{D1D2Ddekinji}, we obtain \eqref{Z1refes}, \eqref{Z2refes}, and \eqref{Z0refes}.
\end{proof}
Now we prove that $z_3$ converges. Essentially, if three points form a triangular shape, the angle corresponding to opposite signs will close, the distance between solitons of the same sign will decrease, and \eqref{Vrep} will not be valid. Before proceeding to the proof, we define 
\begin{align}\label{theta4}
\theta_4=\frac{\theta_1-1}{4}.
\end{align}
\begin{lemma}\label{modifiedZ}
Let $\vec{u}$ be a solution of \eqref{DNKG} satisfying \eqref{sa} and \eqref{3soliass} for some $z,\ \sigma$, and $0<\delta<\delta_3$. Then there exists $z_{\infty}\in \mathbb{R}^d$ such that
\begin{align}\label{z3es}
|z_3(t)-z_{\infty}|\lesssim t^{-\theta_4}.
\end{align}
Furthermore, we have
\begin{align}\label{D0esnew}
\mathcal{F}(D_0)\lesssim t^{-\frac{3}{2}}.
\end{align}
\end{lemma}
\begin{proof}
First, by Lemma \ref{3solirep} and Lemma \ref{epzes}, we have
\begin{align*}
\left|\dot{z}_1+\dot{z}_2+\dot{z}_3\right|\lesssim e^{-\theta_1 D}\lesssim t^{-\theta_3}.
\end{align*}
Thus there exists $z_a\in \mathbb{R}^d$ such that 
\begin{align}\label{jushines}
\left|(z_1(t)+z_2(t)+z_3(t))-z_a\right|\lesssim t^{-\theta_3+1}\lesssim t^{-\theta_4}.
\end{align}
Second, we prove
\begin{align}\label{mannnakaes}
\left|z_1(t)+z_2(t)-2z_3(t)\right|\lesssim t^{-\theta_4}.
\end{align}
By Lemma \ref{3soliZesrefine}, when we define
\begin{align*}
W=Z_1+Z_2,
\end{align*}
we have
\begin{align}\label{Wes}
\left|\dot{W}-\frac{3\mathcal{F}(D)}{D}W\right|\lesssim t^{-\theta_3}.
\end{align}
Furthermore, we have
\begin{align}\label{Wes2}
\left| \frac{d}{dt}|W|-\frac{3\mathcal{F}(D)}{D}|W|\right|\lesssim t^{-\theta_3}.
\end{align}
By  \eqref{logtorder}, we have $\mathcal{F}(D)\sim t^{-1}$ and therefore we have $D\sim \log{t}$. Thus, there exists $c>0$ such that for $t>2$
\begin{align}\label{Wdekai}
\frac{d}{dt}|W|\geq \frac{c}{t\log{t}}|W|-c^{-1}t^{-\theta_3}.
\end{align}
In particular, by \eqref{Wdekai}, we have for $t>s>2$
\begin{align}\label{Gronwalles}
|W(t)|\geq (\log{t})^c\left(\frac{ |W(s)|}{(\log{s})^c}-Cs^{-\theta_2}\right)
\end{align}
for some $C>0$ independent of $T>0$. We assume that there exists $T\gg 1$ such that 
\begin{align*}
|W(T)|>T^{-\theta_4}>2CT^{-\theta_2}(\log{T})^c.
\end{align*}
Then, by \eqref{Gronwalles}, we obtain $|W(t)|\to \infty$ as $t\to\infty$. We define 
\begin{align*}
\mathcal{X}(t)=\frac{D_0(t)}{|W(t)|}.
\end{align*}
Then by Lemma \ref{3soliZesrefine}, we have
\begin{align*}
\left|\frac{d}{dt}D_0-\left(\frac{\mathcal{F}(D)}{D}-\frac{2\mathcal{F}(D_0)}{D_0}\right)D_0\right|\lesssim t^{-\theta_3},
\end{align*}
and therefore we have
\begin{align}\label{mathcalX}
\begin{aligned}
\frac{d}{dt}\mathcal{X}&=\frac{1}{|W(t)|^2}\left(\left(\frac{d}{dt}D_0(t)\right)|W(t)|-\left(\frac{d}{dt}|W(t)|\right)D_0(t)\right)
\\
&=-\frac{2D_0(t)}{|W(t)|}\left(\frac{\mathcal{F}(D(t))}{D(t)}+\frac{\mathcal{F}(D_0(t))}{D_0(t)}\right)+O\left(\frac{(|W(t)|+D_0(t))t^{-\theta_3}}{|W(t)|^2}\right)
\\
&=-\left(\frac{2\mathcal{F}(D(t))}{D(t)}+\frac{2\mathcal{F}(D_0(t))}{D_0(t)}\right)\mathcal{X}(t)+O(t^{-\theta_4-1}).
\end{aligned}
\end{align}
We note that the bottom of \eqref{mathcalX} follows from 
\begin{align*}
 \frac{(|W(t)|+D_0(t))t^{-\theta_3}}{|W(t)|^2}\lesssim D_0(t)t^{-\theta_3}\lesssim D(t)t^{-\theta_3}\lesssim t^{-\theta_3}(\log{t})\lesssim t^{-\theta_4-1}.
 \end{align*}
Thus there exists $c_1>0$ such that 
\begin{align*}
\frac{d}{dt}\mathcal{X}(t)\leq -\frac{c_1}{t\log{t}}\cdot \mathcal{X}(t)+c_1^{-1}t^{-\theta_4-1}.
\end{align*}
Then by the Gronwall's inequality, we obtain $\lim_{t\to\infty}\mathcal{X}(t)=0$. Then, for large $t$, we have
\begin{align*}
D_0(t)<\frac{1}{3}|W(t)|=\frac{1}{3}|Z_1(t)+Z_2(t)|\leq \frac{1}{3}(D_1(t)+D_2(t))<D(t),
\end{align*}
which contradicts Lemma \ref{repcondlem}. Therefore, there exists $T>0$ such that for $t>T$
\begin{align*}
|W(t)|=|z_1(t)+z_2(t)-2z_3(t)|\lesssim t^{-\theta_4},
\end{align*}
which implies \eqref{mannnakaes}. By \eqref{jushines} and \eqref{mannnakaes}, we obtain \eqref{z3es}, where we define $z_{\infty}=\frac{z_a}{3}$. Finally, we show \eqref{D0esnew}. By Lemma \ref{Dnosalem}, we have
\begin{align*}
D_0=|z_1-z_2|=|2(z_1-z_3)-(z_1+z_2-2z_3)|\geq 2D_1-|W|\geq \frac{5}{3}D,
\end{align*}
and we obtain \eqref{D0esnew}. So we complete the proof.
\end{proof}

\subsection{Proof of Theorem \ref{maintheorem}}
Here, we prove Theorem \ref{maintheorem}. First, we show the following lemma.
\begin{lemma}\label{mainlemma}
Let $\vec{u}$ be a solution of \eqref{DNKG} satisfying \eqref{sa} and \eqref{3soliass} for some $z,\ \sigma$, and $0<\delta<\delta_3$. Then there exist $c_0>0$ depending only on $d,\ p,\ \alpha$, and  $z_{\infty}\in \mathbb{R}^d$ and $\omega_{\infty}\in S^{d-1}$ such that 
\begin{align}
\left|z_1(t)-z_{\infty}-\omega_{\infty}\left(\log{t}-\frac{d-1}{2}\log{(\log{t})}+c_0\right)\right|&\lesssim \frac{\log{(\log{t})}}{\log{t}}, \label{modifiedz1}
\\
\left|z_2(t)-z_{\infty}+\omega_{\infty}\left(\log{t}-\frac{d-1}{2}\log{(\log{t})}+c_0\right)\right|&\lesssim \frac{\log{(\log{t})}}{\log{t}}, \label{modifiedz2}
\\
\left|z_3(t)-z_{\infty}\right|&\lesssim t^{-\theta_4}. \label{modifiedz3}
\end{align}
\end{lemma}

\begin{proof}
First, by Lemma \ref{modifiedZ}, we have \eqref{modifiedz3}. By Lemma \ref{3solirep}, $\vec{u}$ satisfies \eqref{Vrep}. Then by Lemma \ref{epzes} and Lemma \ref{modifiedZ}, we have
\begin{align}\label{z1ode}
\left|\dot{z}_1-\mathcal{F}(|z_1-z_{\infty}|)\frac{z_1-z_{\infty}}{|z_1-z_{\infty}|}\right|\lesssim t^{-\theta_4-1}.
\end{align}
By $\mathcal{F}(D)\sim t^{-1}$ and \eqref{z1ode}, we obtain
\begin{align*}
\left|\frac{\frac{d}{dt}|z_1-z_\infty|}{\mathcal{F}(|z_1-z_{\infty}|)}-1\right|\lesssim t^{-\theta_4}.
\end{align*}
Thus, there exists $c_0\in \mathbb{R}$ depending only on $d,\ p,\ \alpha$ such that
\begin{align}\label{Dsharpes}
\left||z_1(t)-z_{\infty}|-\left(\log{t}-\frac{d-1}{2}\log{(\log{t})}+c_0\right)\right|\lesssim \frac{\log{(\log{t})}}{\log{t}}.
\end{align}
Furthermore, by \eqref{Dsharpes}, we have
\begin{align*}
\left|\frac{d}{dt} \left(\frac{z_1-z_{\infty}}{|z_1-z_{\infty}|}\right)\right|\lesssim t^{-\theta_4-1},
\end{align*}
and therefore, there exists $\omega_{\infty}\in S^{d-1}$ such that 
\begin{align}\label{omegaconv}
\left| \frac{z_1-z_{\infty}}{|z_1-z_{\infty}|}-\omega_{\infty}\right|\lesssim t^{-\theta_4}.
\end{align}
By \eqref{Dsharpes} and \eqref{omegaconv}, we obtain \eqref{modifiedz1}. Last, by \eqref{mannnakaes}, \eqref{modifiedz1} and \eqref{modifiedz3}, we obtain \eqref{modifiedz2}. Therefore we complete the proof.
\end{proof}
Here, we prove Theorem \ref{maintheorem}.
\begin{proof}[Proof of Theorem \ref{maintheorem}]
If $\vec{u}$ is a $3$-solitary waves, by \cite[Theorem 1.8]{CD}, we may assume $\sigma_1=\sigma_2=-\sigma_3$. Then by \eqref{Ksoli} and Lemma \ref{modKsoli}, there exists $T>0$ such that for $t>T$, \eqref{modlim1}-\eqref{modeq} hold true for some $\tilde{z}$. Furthermore, we may assume for $t>T$
\begin{align*}
\|\vec{\varepsilon}\|_{\mathcal{H}}+\frac{1}{D(t)}<\delta<\delta_3.
\end{align*}
Here, we define 
\begin{align*}
\vec{\tilde{u}}(t)=\vec{u}(t+T).
\end{align*}
Then by \eqref{invariance}, $\vec{\tilde{u}}$ satisfies \eqref{sa} and \eqref{3soliass} for some $\tilde{z},\ \sigma$, and $0<\delta<\delta_3$. Therefore, by \eqref{invariance} and Lemma \ref{mainlemma}, there exist $\tilde{z}_{\infty}$ and $\tilde{\omega}_{\infty}$ such that 
\begin{align*}
\left|\tilde{z}_1(t+T)-\tilde{z}_{\infty}-\tilde{\omega}_{\infty}\left(\log{t}-\frac{d-1}{2}\log{(\log{t})}+c_0\right)\right|&\lesssim \frac{\log{(\log{t})}}{\log{t}}, 
\\
\left|\tilde{z}_2(t+T)-\tilde{z}_{\infty}+\tilde{\omega}_{\infty}\left(\log{t}-\frac{d-1}{2}\log{(\log{t})}+c_0\right)\right|&\lesssim \frac{\log{(\log{t})}}{\log{t}}, 
\\
\left|\tilde{z}_3(t+T)-\tilde{z}_{\infty}\right|&\lesssim t^{-\theta_4},
\end{align*} 
and therefore we have
\begin{align*}
\left|\tilde{z}_1(t)-\tilde{z}_{\infty}-\tilde{\omega}_{\infty}\left(\log{t}-\frac{d-1}{2}\log{(\log{t})}+c_0\right)\right|&\lesssim \frac{\log{(\log{t})}}{\log{t}}, 
\\
\left|\tilde{z}_2(t)-\tilde{z}_{\infty}+\tilde{\omega}_{\infty}\left(\log{t}-\frac{d-1}{2}\log{(\log{t})}+c_0\right)\right|&\lesssim \frac{\log{(\log{t})}}{\log{t}}, 
\\
\left|\tilde{z}_3(t)-\tilde{z}_{\infty}\right|&\lesssim t^{-\theta_4}.
\end{align*}
In addition, by the proof of Lemma \ref{modKsoli}, we have
\begin{align*}
|z(t)-\tilde{z}(t)|\lesssim \|\vec{\varepsilon}(t)\|_{\mathcal{H}}.
\end{align*}
Furthermore, by Lemma \ref{epzes} and \eqref{logtorder}, we have
\begin{align*}
\|\vec{\varepsilon}(t)\|_{\mathcal{H}}+|z(t)-\tilde{z}(t)|\lesssim \frac{1}{t}.
\end{align*}
Gathering these arguments, we obtain \eqref{3soli} for $\tilde{z}$ and $\tilde{z}$ satisfies \eqref{zes} for $\tilde{z}_{\infty}\in \mathbb{R}^d$ and $\tilde{\omega}_{\infty}\in S^{d-1}$. Thus we complete the proof.
\end{proof}

\section*{Acknowledgement}

The author wishes to thank Kenji Nakanishi for many helpful comments, discussions and encouragement on the present paper.


\begin{thebibliography}{99}
%
% The \bibitem commands: 
% Please follow "Notice to Authors" for referencing.  You 
% must specify bold and italic fonts yourself. 
%

\bibitem{A} S. Aryan,
\textit{Existence of two-solitary waves with logarithmic distance for the nonlinear Klein-Gordon equation,}
Commun. Contemp. Math. \textbf{24} (2022) no.1, 2050091


\bibitem{A1} S. Aryan,
\textit{Soliton resolution for the energy-critical nonlinear heat equation in the radial case},
preprint, arXiv:2405.06005.


\bibitem{BRS} N. Burq, G. Raugel, W. Schlag,
\textit{Long time dynamics for damped Klein-Gordon equations},
Ann. Sci. \'{E}c. Norm. Sup\'{e}r. (4) \textbf{50} (2017), no.6, 1447-1498.


%\bibitem{C} T. Cazenave, 
%\textit{Uniform estimates for solutions of nonlinear Klein-Gordon equations},
%J. Funct. Anal. \textbf{60} (1985), no.1, 36-55.


\bibitem{CH} T. Cazenave, A. Haraux, 
\textit{An introduction to semilinear evolution equations},
Oxford Lecture Series in Mathematics ans its Applications {\bf13}, Oxford University Press, 1988. 


\bibitem{CGNT} S. M. Chang, S. Gustafson, K. Nakanishi, T. P. Tsai,
\textit{Specrta of linearized operators for NLS solitary waves.}
SIAN J. Math. Anal. \textbf{39} (2007/08), no.4, 1070-1111.

\bibitem{CDKM} C. Collot, T. Duyckaerts, C. Kenig, F. Merle,
\textit{Soliton Resolution for the Radial Quadratic Wave Equation in Space Dimension 6},
Vientnam J. Math. \textbf{52} (2024), 735-773.

\bibitem{CD} R. C\^{o}te, H. Du,
\textit{Construction of multi solitary waves with symmetry for the damped nonlinear Klein-Gordon equation},
preprint, arXiv:2411.11703.




\bibitem{CM1} R. C\^{o}te, Y. Martel,
\textit{Multi-traveling waves for the nonlinear Klein-Gordon equations},
Trans. Amer. Math. Soc. \textbf{370} (2018), no.10, 7461-7487.




\bibitem{CMY} R. C\^{o}te, Y. Martel, X. Yuan, 
\textit{Long time asymptotics of the one-dimensional damped nonlinear Klein-Gordon equation},
Arch. Ration. Mech. Anal. {\bf 239} (2021), no.3, 1837-1874.



\bibitem{CMYZ} R. C\^{o}te, Y. Martel, X. Yuan, L. Zhao,
\textit{Description and classification of 2-solitary waves for nonlinear damped Klein-Gordon equations},
Comm. Math. Phys. \textbf{388} (2021), 1557-1601.



\bibitem{CM} R. C\^{o}te, C. Mun\~{o}z,
\textit{Multi-solitons for nonlinear Klein-Gordon equations},
Forum Math. Sigma \textbf{2} (2014), Paper No.e15, 38pp.

%\bibitem{CGOR} E. Csobo, F. Genoud, M.Ohta, J. Royer,
%\textit{Stability of standing waves for a nonlinear Klein-Gordon equation with delta potentials},
%J. Differential Equations \textbf{268} (2019), no. 1, 353-388.


\bibitem{CY} R. C\^{o}te, X. Yuan,
\textit{Asymptotics of solutions with a compactness property for the nonlinear damped Klein-Gordon equation},
Nonlinear Anal. \textbf{218} (2022), Paper No. 112768, 34pp.

\bibitem{DJKM} T. Duyckaerts, H. Jia, C. Kenig, F. Merle,
\textit{Soliton resolution along a sequence of times for the focusing energy critical wave equation},
Geom. Funct. Anal. \textbf{27} (2017), 798-862.

\bibitem{DKMM} T. Duyckaerts, C. Kenig, Y. Martel, F. Merle,
\textit{Soliton resolution for critical co-rotational wave maps and radial cubic wave equation},
Comm. Math. Phys. \textbf{391} (2022), no.2, 779-871.


\bibitem{DKM1} T. Duyckaerts, C. Kenig, F. Merle,
\textit{Classification of radial solutions of the focusing, energy-critical wave equation},
Camb. J. Math \textbf{1}(1) (2013), 75-144.

%\bibitem{DKM3} T. Duyckaerts, C. Kenig, F. Merle,
%\textit{Exterior energy bounds for the critical wave equation close to the ground state},
%Comm. Math. Phys. \textbf{379} (2020) 1113-1175.

%\bibitem{DKM4} T. Duyckaerts, C. Kenig, F. Merle,
%\textit{Decay estimates for nonradiative solutions of the energy-critical focusing wave equations},
%J. Geom. Anal (2021)

\bibitem{DKM2} T. Duyckaerts, C. Kenig, F. Merle,
\textit{Soliton resolution for the radial critical wave equation in all odd space dimensions},
Acta Math. \textbf{230} (2023), no. 1, 1-92.



\bibitem{F} E. Feireisl,
\textit{Finite energy travelling waves for nonlinear damped wave equations}, Quart. Appl. Math. \textbf{56} (1998), 5570.



%\bibitem{FJ} R. Fukuizumi, L. Jeanjean, 
%\textit{Stability of standing waves for a nonlinear Schr\"{o}dinger equation with a repulsive Dirac delta potential}, Discrete Contin. Dyn. Syst. \textbf{21} (2008), no.1, 121-136.

%\bibitem{FOO} R. Fukuizumi, M. Ohta, T. Ozawa, 
%\textit{Nonlinear Sch\"{o}dinger equation with a point defect},
%Ann. Inst. H. Poincar\'{e} Anal. Non Lin\'{e}aire \textbf{25} (2008) 837-845.


\bibitem{GI} S. Gustafson, T, Inui,
\textit{Two-solitons with logarithmic separation dor 1D NLS with repulsive delta potential},
preprint, arXiv:2310.08865.




%\bibitem{GIS} S. Gustafson, T. Inui, I, Shimizu,
%\textit{Multi-solitons for the nonlinear Schr\"{o}dinger equation with repulsive dirac delta potential},
%preprint, arXiv:2310.08862.


\bibitem{GZ} J. Gu, L. Zhao,
\textit{Soliton resolution for the energy critical damped wave equations in the radial case},
preprint, arXiv:2401.04114.


%\bibitem{IMN} S. Ibrahim, N. Masmoudi, K. Nakanishi,
%\textit{Scattering threshold for the focusing nonlinear Klein-Gordon equation},
%Anal. PDE \textbf{4} (2011), no.3, 405-460



%\bibitem{II} M. Ikeda, T. Inui,
%\textit{Global dynamics below the standing waves for the focusing semilinear Schrödinger equation with a repulsive Dirac delta potential},
%Anal. PDE \textbf{10} (2017), no.2, 481-512


\bibitem{I} K. Ishizuka,
\textit{Long-time asymptotics of the damped nonlinear Klein-Gordon equation with a delta potential},
Nonlinear Anal. \textbf{253} (2025), no.113732. 




\bibitem{IN} K. Ishizuka, K. Nakanishi,
\textit{Global dynamics around 2-solitons for the nonlinear damped Klein-Gordon equations},
Ann. PDE \textbf{9} (2023), no.1, Paper No. 2, 79pp.

%\bibitem{JT} L. Jeanjean, K. Tanaka,
%\textit{A positive solution for a nonlinear Sch\"{o}dinger equation on $\mathbb{R}^N$},
%Indiana. Univ. Math. J. \textbf{54} (2005), 443-464.




\bibitem{J1} J. Jendrej,
\textit{Construction of two-bubble solutions for the energy-critical NLS},
Anal. PDE \textbf{10}(8) (2017), 1923-1959.

\bibitem{J2} J. Jendrej, 
\textit{Construction of two-bubble solutions for energy-critical wave equations},
Am. J. Math. \textbf{141}(1) (2019), 55-118.


\bibitem{JL1} J. Jendrej, A. Lawrie,
\textit{Two-bubble dynamics for threshold solutions to the wave maps equation},
Invent. Math. \textbf{213}(3) (2018), 1249-1325.


\bibitem{JL2} J. Jendrej, A. Lawrie,
\textit{An asymptotic expansion of two-bubble wave maps in high equivariance classes},
Anal. PDE \textbf{15} (2022), no.2, 327-403.

\bibitem{JL4} J. Jendrej, A. Lawrie,
\textit{Soliton resolution for the energy-critical nonlinear wave equation in the radial case},
Ann. PDE \textbf{9} (2023) no.2, Paper No. 18, 117pp.


\bibitem{JL3} J. Jendrej, A. Lawrie,
\textit{Soliton resolution for energy-critical wave maps in the equivariant case},
preprint, arXiv:2106.10738.

\bibitem{JLX} H. Jia, B. Liu, G. Xu,
\textit{Long time dynamics of defocusing energy critical 3+1 dimensional wave equation with potential in the radial case},
Commun. Math. Phys. \textbf{339}(2) (2015), 353-384.


\bibitem{K1} C. Keller,
\textit{Stable and unstable manifolds for the nonlinear wave equation with dissipation},
J. Differ. Equ. \textbf{50}(3), (1983), 330-347.


\bibitem{KK} T. Kim, S. Kwon,
\textit{Soliton resolution for Calgero-Moser derivative nonlinear Sch\"{o}dinger equation},
preprint, arXiv:2408.12843.

\bibitem{K2} M. Kwong,
\textit{ Uniqueness of positive solutions of $\Delta u-u+u^p=0$ in ${\mathbb{R}}^N$},
Arch. Ration. Mech. Anal. \textbf{105} (1989), 243-266.


\bibitem{L} P. L. Lions, 
\textit{Solutions of Hartree-Fock equations for Coulomb systems},
Comm. Math. Phys. \textbf{109} (1987), 33-97.


\bibitem{L2} P. L. Lions,
\textit{On positive solutions of semilinear elliptic equations in unbounded domains},
Math. Sci. Res. Inst. Publ., \textbf{13}, Springer New York, 1998.


\bibitem{LMZ} X.Li, C. Miao, L. Zhao,
\textit{Soliton resolution for the energy critical wave equation with inverse-square potential in the radial case},
preprint, arXiv:2201.12957.



\bibitem{MN} Y. Martel, T. V. Nguyen,
\textit{Construction of 2-solitons with logarithmic distance for the one-dimensional cubic Schr\"{o}dinger system},
Discrete Contin. Dyn. Syst. \textbf{40} (2020), no.3, 1595-1620.

\bibitem{N1} T. V. Nguyen,
\textit{Strongly interacting multi-solitons with logarithmic relative distance for the gKdV equation},
Nonlinearity \textbf{30} (2017), no.12, 4614-4648.


\bibitem{N2} T. V. Nguyen,
\textit{Existence of multi-solitary waves with logarithmic relative distances for the NLS equation},
C. R. Math. Acad> Sci. Paris \textbf{357} (2019), no.1, 13-58.


%\bibitem{NS}
%K. Nakanishi, W. Schlag,
%\textit{Global dynamics above the ground state energy for the cubic NLS equation in 3D}, 
%Calc. Var. Partial Differential Equations 
%{\bf 44} (2012), no. 1-2, 1--45.

\bibitem{NS}
K. Nakanishi, W. Schlag, 
\textit{Global dynamics above the ground state for the nonlinear Klein-Gordon equation without a radial assumption}, 
Arch. Ration. Mech. Anal. 
{\bf 203} (2012), no. 3, 809--851. 

%\bibitem{PS}
%L. E. Payne, D. H. Sattinger,
%\textit{Saddle points and instability of nonlinear hyperholic equations},
%Israel J. Math. \textbf{22} (1975), 272-303.

%\bibitem{TX} X.Tang, G. Xu,
%\textit{Minimal mass blow-up solutions for the $L^2$-critical NLS with the delta potential for even data in one dimension},
%SIAM J. Math. Anal. \textbf{56} (2024), no.2, 1727-1769.



%\bibitem{Wei89}
%M. Weinstein, 
%\textit{The nonlinear Schr\"{o}dinger equation---singularity formation, stability and dispersion}, 
%Contemp. Math. {\bf 99} (1989), 213--232.


\end{thebibliography}
\end{document}